\documentclass[11pt]{article}
\usepackage{definitions}
\usepackage{dsfont}

\begin{document}

\title{On the Convergence of L-shaped Algorithms for Two-Stage Stochastic Programming}

\author{John R. Birge\footnote{The University of Chicago, Booth School of Business (john.birge@chicagobooth.edu).} \and Haihao Lu\footnote{The University of Chicago, Booth School of Business (haihao.lu@chicagobooth.edu).} \and Baoyu Zhou\footnote{University of Michigan, Department of Industrial and Operations Engineering (zbaoyu@umich.edu).}}
\date{}
\maketitle

\begin{abstract}
  In this paper, we design, analyze, and implement a variant of the two-loop L-shaped algorithms for solving two-stage stochastic programming problems that arise from important application areas including revenue management and power systems. We consider the setting in which it is intractable to compute exact objective function and (sub)gradient information, and instead, only estimates of objective function and (sub)gradient values are available. Under common assumptions including fixed recourse and bounded (sub)gradients, the algorithm generates a sequence of iterates that converge to a neighborhood of optimality, where the radius of the convergence neighborhood depends on the level of the inexactness of objective function estimates. The number of outer and inner iterations needed to find an approximate optimal iterate is provided. More specifically, suppose $\beps$ is the error level of the function evaluation, we show that the L-shaped method with constant step-size can find an $\beps$-optimal solution within $O(\frac{1}{\beps^3})$ total iterations. Suppose the optimal objective is known, we propose a Polyak-type step-size that can find an $\beps$-optimal solution within $O(\frac{1}{\beps^2})$ total iterations. It turns out that the rate $O(\frac{1}{\beps^2})$ also matches with the complexity lower bound for first-order non-smooth optimization even in the exact situation.
Furthermore, assuming the problem satisfies the sharpness condition, which holds when the support of the scenario is finite, then we can improve the complexity of the algorithm with constant step-size to $O(\frac{1}{\beps})$ and Polyak-type step-size to $O\pran{\log\pran{\frac{1}{\beps}}}$. Finally, we show a sample complexity result for the algorithm with a Polyak-type step-size policy that can be extended to analyze other situations. We also present a numerical study that verifies our theoretical results and demonstrates the superior empirical performance of our proposed algorithms over classic solvers.
\end{abstract}

\section{Introduction}
Two-stage stochastic programming is a fundamental class of optimization problems formulated as
\begin{equation}\label{prob.main}
\begin{aligned}
\min_{x\in X}\ F(x) + \mathbb{E}_{\xi} \left[\min_{y(\xi)\in Y(x;\xi)}\ Q(x,y(\xi);\xi)\right], 
\end{aligned}
\end{equation}
where $\mathbb{E}_{\xi}[\cdot]$ represents the expectation taken with respect to the distribution of random variable $\xi$, and decision variables $(x,y(\xi))$ are chosen from feasible sets $(X,Y(x;\xi))$ respectively. Problem \eqref{prob.main} considers decision variables from two different stages: the first-stage decision variable $x\in X$ is determined before any realization of random variable $\xi$, while the second-stage decision variable $y(\xi)$ is evaluated after observing a realization of $\xi$. The goal of two-stage stochastic programming in the form of \eqref{prob.main} is to optimize the combination of objective functions from both stages. 
This type of problems has numerous practical applications, such as portfolio selection (\cite{ShapDentRusz21,DantInfa93}), manufacturing resource allocation (\cite{BirgLouv11}), inventory control (\cite{DillOlivAbba17}), supply chain management (\cite{MaruEksiHuan14}), disaster management (\cite{Noya12}), water resource management (\cite{HuanLouc00}), and transportation network protection (\cite{LiuFanOrdo09}), just to name a few. In many applications, the objective function in the first stage, $F(x)$, is a linear function, and the problem in the second stage $Q(x,y(\xi);\xi)$ can be formulated as linear programming, in which case, the problem is called two-stage stochastic linear programming (two-stage SLP).

For example, in inventory management problems, the decision maker aims to minimize the overall net cost (or equivalently, maximize the overall net income) by buying commodities from factories and selling them to customers at a higher price. This whole process involves buying cost, holding cost, and selling profit, and such problem could be formulated as a two-stage stochastic programming; e.g.,
\begin{equation*}
\begin{aligned}
\min_{x\in\mathbb{R}^n}\ &p^Tx + h^Tx + \mathbb{E}_{\xi}\left[\min_{y(\xi)\in Y(x;\xi)}\ Q(x,y(\xi);\xi)\right] \\
\text{s.t.}\ &p^Tx \leq b, \quad x\geq 0
\end{aligned}
\end{equation*}
and
\begin{equation*}
\begin{aligned}
\min_{y(\xi)\in Y(x;\xi)}\ Q(x,y(\xi);\xi) = \min_{y(\xi)\in\mathbb{R}^n}\ &-s(\xi)^Ty(\xi) \\
\text{s.t.}\ &0\leq y(\xi)\leq x \\
 &y(\xi)\leq r\cdot c(\xi),
\end{aligned}
\end{equation*}
where $p\in\mathbb{R}^n$, $h\in\mathbb{R}^n$, $b\in\mathbb{R}$, $s(\xi)\in\mathbb{R}^n$, $r\in(0,1)$ and $c(\xi)\in\mathbb{R}^n$ are all parameters. In this two-stage stochastic programming, the first-stage decision variable $x$ represents the number of commodities bought from factories, while the first-stage feasible set $X:=\{x\in\mathbb{R}^n:p^Tx \leq b \text{ and } x \geq 0\}$ includes two classes of constraints: the non-negativity constraint $x\geq 0$ and the budget constraint $p^Tx \leq b$, where $p\in\mathbb{R}^n$ means the per-unit buying cost for $n$ different commodities and $b\in\mathbb{R}$ is the total budget. 
By defining $h\in\mathbb{R}^n$ as the per-unit holding cost for $n$ different commodities, the first-stage objective function $(p^Tx + h^Tx)$, which can be seen as $F(x)$ in \eqref{prob.main}, means the total buying and holding costs for $n$ commodities.
In terms of the second-stage problem, we have $\xi$ as the randomness in the nature, such as weather, time, and date, that affects the number of customers visiting the store to buy commodities, and the decision variable $y(\xi)$ as the number of commodities sold to customers under scenario $\xi$. Meanwhile, the second-stage feasible set $Y(x;\xi) := \{y(\xi)\in\mathbb{R}^n: 0\leq y(\xi)\leq x \text{ and } y(\xi) \leq r\cdot c(\xi)\}$ includes two natural constraints: (i) $0\leq y(\xi)\leq x$ means the number of commodities sold to customers must be non-negative and does not exceed what have been bought from factories, and (ii) $y(\xi)\leq r\cdot c(\xi)$, where $r\in(0,1)$ is a ratio parameter and $c(\xi)$ represents the number of customers expected to be visiting the store under scenario $\xi$. Finally, with $s(\xi)\in\mathbb{R}^n$ being the per-unit selling price of commodities under scenario $\xi$, we may set $Q(x,y(\xi);\xi) = -s(\xi)^Ty(\xi)$ to minimize the total cost (equivalently, maximizing the total profit) in the second stage.
In the aforementioned example, under any scenario $\xi$, we always have constraint sets as polytopes and objective functions being linear for both stages, which turns the problem into a two-stage stochastic linear programming. 

Many optimization methods have been developed for solving two-stage stochastic programming~\eqref{prob.main}. The most popular approach may be reformulating \eqref{prob.main} as large-scale linear programming (LP) by explicitly writing down the expectation term for the second stage problem. Such reformulation works only when the number of scenarios is finite. However, due to the curse of dimensionality, directly solving the LP reformulation may quickly become intractable as the size of the problem gets larger. As such, decomposition techniques have been proposed to solve this large LP by taking advantage of its structures. In general, the decomposition techniques could be divided into two classes: primal decomposition and dual composition. Primal decomposition methods decompose the LP reformulation over stages, while dual decomposition methods deal with the reformulation over scenarios (\cite{RosaRusz96}). Among these decomposition techniques, the L-shaped algorithm, a class of Benders decomposition method (\cite{Bend62}), originally proposed by \cite{VanWets69}, is particularly well-known due to its practical performance in solving large-scale two-stage stochastic programming problems. In order to make sure that the sub-problem in the L-shaped algorithm has a finite solution, the regularized L-shaped method was proposed (\cite{Rusz86}), which is now a base algorithm used in two-stage stochastic programming solvers~(\cite{BielJoha22}). For many large-scale two-stage stochastic programming problems, it is prohibitive to compute true function and (sub)gradient values due to the problem size and nature of stochasticity. Therefore, in this paper, we focus on an inexact regularized L-shaped algorithm, which only evaluates inexact function and (sub)gradient estimates. \cite{Lema86,Kiwi90} showed that such an inexact regularized L-shaped algorithm is equivalent to an inexact proximal bundle method in nonsmooth optimization, which was later on shown to asymptotically converge to a near-optimal solution~(\cite{OlivSagaSche11,OlivSagaLema14}). A long-existing open problem is whether there is a complexity theory for the inexact regularized L-shaped method (and for the inexact proximal bundle method), i.e., how many iterations are needed to identify a solution with a certain quality. In this paper, we answer this open question by presenting the complexity theory of inexact regularized L-shaped algorithms under different assumptions. More specifically, the contributions of this paper can be summarized as:

\begin{enumerate}
    \item We present the first complexity result for the inexact regularized L-shaped method. Suppose the error in the function evaluation is $\beps$. We study the quality of the solution that the inexact regularized L-shaped algorithm can find, and the total outer and inner iterations needed before such a solution can be found. We present the complexity results for inexact regularized L-shaped algorithm with two step-size rules: a basic constant step-size policy (Section \ref{sec.const_step_policy}), and an optimal step-size policy (Section \ref{sec:ideal_ss}), which mimics Polyak's step-size rule for non-smooth optimization. Furthermore, if the number of scenarios is finite (i.e., the random variable $\xi$ has finite support), the problem~\eqref{prob.main} satisfies a sharpness condition, under which the complexity results can be further improved. These complexity results are summarized in Table \ref{table.conv_summary}.
    \item Under proper regularity conditions, we present the sample complexity result of the inexact regularized L-shaped method (Section \ref{sec.sample_const}), namely, the number of total samples needed in order to find a solution that is near-optimal 
    with high probability.
    \item We present a numerical study (Section \ref{sec.num_results}), which not only validates our theoretical results but also
    showcases that the inexact regularized L-shaped algorithm with our proposed parameters has superior performance compared to the state-of-the-art solvers for two-stage SLP for large-scale instances.
\end{enumerate}

\begin{table}[]
\caption{Summary of the complexity results for the regularized L-shaped methods under different regularity conditions and step-size rules, where $\beps$ is the error level in the inexact function evaluation.}\label{table.conv_summary}
\centering
\begin{tabular}{|c|c|c|c|}
\hline
Problem Setting & Optimality Gap & $\#$ of Outer Iterations & $\#$ of Inner Iterations  \\ \hline
\begin{tabular}[c]{@{}c@{}} Without sharpness; \\ constant step size \end{tabular} & $\mathcal{O}(\beps)$ & $\mathcal{O}\left(\frac{1}{\beps}\right)$ & $\mathcal{O}\left(\frac{1}{\beps^3}\right)$ \\ \hline
\begin{tabular}[c]{@{}c@{}} Without sharpness; \\ optimal step size \end{tabular} & $\mathcal{O}(\beps)$ & $\mathcal{O}\left(\log\left(\frac{1}{\beps}\right)\right)$ & $\mathcal{O}\left(\frac{1}{\beps^2}\right)$  \\ \hline
\begin{tabular}[c]{@{}c@{}} With sharpness; \\ constant step size \end{tabular} & $\mathcal{O}(\beps)$ & $\mathcal{O}\left(\frac{1}{\beps}\right)$ & $\mathcal{O}\left(\frac{1}{\beps}\right)$ \\ \hline
\begin{tabular}[c]{@{}c@{}} With sharpness; \\ optimal step size \end{tabular} & $\mathcal{O}(\beps)$ & $\mathcal{O}\left(\log\left(\frac{1}{\beps}\right)\right)$ & $\mathcal{O}\left(\log\left(\frac{1}{\beps}\right)\right)$ \\ \hline
\end{tabular}
\end{table}

\subsection{Literature Review}
One of the most popular methods for solving two-stage SLP problems is by reformulating the problem as a large linear programming (LP) and then solving it either using a general-purpose LP solver (such as Gurobi) or specialized algorithms. Since such LPs are usually huge, and share a certain block diagonal structure, decomposition algorithms, including augmented Lagrangian decomposition (\cite{MulvRusz95,Rusz95}), Dantzig-Wolfe decomposition (\cite{DantWolf61}), Benders decomposition (\cite{Bend62}), have been proposed to solve large-scale instances. 
In particular, the Benders decomposition (\cite{Bend62}) to a primal problem is equivalent to the Dantzig-Wolfe decomposition (\cite{DantWolf61}) to the corresponding dual problem (\cite{DantThap03}).
Both methods have obtained numerical success for solving large-scale two-stage SLP~(\cite{RosaRusz96}).  
Other than decomposition methods, \cite{Kall79} proposes a revised simplex method with reduced computational cost for solving two-stage SLP problems. More recently, \cite{lubin2013parallel} presents a parallel distributed-memory simplex method that solved such an LP with 400 million variables on 1024 computing nodes.

The L-shaped algorithm (\cite{VanWets69}), as a special case of Benders decomposition, is well-known for its outstanding empirical performance in solving large-scale two-stage stochastic programming problems. When solving convex problems, the L-shaped algorithm generates a piecewise linear approximation function that performs as a lower bound to the original problem at every iteration. 
Based on the naive L-shaped algorithm framework, many advanced decomposition and regularization techniques, such as level-set decomposition (\cite{FabiSzok07}), regularized decomposition (\cite{Rusz86}, \cite{RuszSwie97}), and trust-region regularization (\cite{LindWrig03}), have been applied for improving its practical performance. 
These aforementioned decomposition techniques are implemented in a solver \texttt{StochasticPrograms.jl}\footnote{https://github.com/martinbiel/StochasticPrograms.jl}, which is an efficient implementation of the L-shaped algorithm in Julia programming language (\cite{BielJoha22}). \cite{LindShapWrig06} also studied the empirical performance of the L-shaped algorithm proposed in \cite{LindWrig03}. Other than the L-shaped algorithm, \cite{ErmoWets88} introduced a stochastic quasi-gradient method, and \cite{HigleSen96} proposed a stochastic decomposition method for solving two-stage SLP. Recently, \cite{LanShap23} discusses multiple reliable numerical algorithms for solving multi-stage convex stochastic optimization problems, and we refer interested readers to the references therein for the recent numerical advances of L-shaped algorithms.

Meanwhile, a two-stage SLP problem can also be seen as a single-stage nonsmooth optimization problem after optimizing the inner problems. From the nonsmooth optimization perspective, the regularized L-shaped algorithm is equivalent to a proximal bundle method; see \cite{Lema86,Kiwi90}. At every iteration, proximal bundle methods build up a piecewise linear lower bound function by generating the first-order Taylor series approximation (i.e., bundle cuts) at sample points, and update iterates via solving a quadratic subproblem. The asymptotic convergence behavior of proximal bundle methods has been well-developed in both convex (\cite{BonnGilbLemaSga06,MakeNeit92,Kiwi06}) and nonconvex settings (\cite{AtenSagaSilvSolo23,FuduGaudGial04,HareSaga10,HareSagaSolo16,NollProtRond08,VlceLuks01}). There was extensive previous literature studied the complexity theory of the proximal bundle method: \cite{AstoFranFuduGorg13,DiazGrim21,DuRusz17,Kiwi00,LianGuigMont22,LianMont21a,LianMont21b,LianMontZhan23}. However, for a large-scale two-stage SLP, it is in general too expensive to construct an exact lower bound function. Thus, the \emph{inexact} regularized L-shaped method, where an inexact lower bound function is constructed by utilizing part of the samples in every iteration, is the variant of L-shaped method that is widely used in practice (\cite{VerwAhmedKleyNemhShap03}).
In this paper, we present the first complexity result of an inexact regularized L-shaped algorithm. The convergence theory of such inexact oracle was studied in \cite{OlivSagaSche11,OlivSagaLema14}. The difference is that \cite{OlivSagaSche11,OlivSagaLema14} only proved the asymptotic convergence properties of inexact proximal bundle algorithms and our work presents the first complexity theory for this algorithm. In particular, this paper presents the number of operations (i.e., inner iterations) needed by an inexact regularized L-shaped algorithm to find near-optimal solutions of two-stage SLP. It is worth mentioning that the complexity theory of interior-point methods and other related algorithms for solving two-stage SLP has also been explored: the iteration complexity of a weighted barrier decomposition method was discussed in \cite{MehrOzev10}; \cite{BirgeQi88} showed that Schur complement could help to lower the operation complexity of interior-point methods for two-stage SLP; \cite{Zhao99} and~\cite{Zhang01} considered operation complexity of interior-point methods for two-stage SLP, while the operation complexity of a log-barrier method with Benders decomposition was developed in \cite{Zhao01}. We note that part of the proof technique in our analysis was inspired by the recent work on the complexity theory developed for the exact proximal bundle method~(\cite{DiazGrim21}). 

Moreover, there are a few publicly available computational packages for solving two-stage SLP. For example, \texttt{PySP} is a stochastic programming tool in Python utilizing the progressive hedging decomposition strategy (\cite{Hart17pyomo}); \texttt{cvxstoc} (\cite{Ali2015disciplined}) is another open-source package in Python for solving convex stochastic programming problems, which is based on the \texttt{cvxpy} package (\cite{DiamBoyd16}) and applies interior-point method and splitting conic solver (\cite{OdonChuPariBoyd16}). AMPL also has its own package for solving stochastic programming problems, e.g., \texttt{FortSP} (\cite{ElliMitrZver10}), which has options for cutting plane and stochastic decomposition methods. Lastly, \texttt{StochasticPrograms.jl} (\cite{BielJoha22}) is a Julia package using L-shaped algorithm for solving two-stage stochastic programming problems.

\subsection{Notation}
Let $\mathbb{R}$ denote the set of real numbers and $\overline{\mathbb{R}}$ denote the extended set of real numbers such that $\overline{\mathbb{R}}:=\mathbb{R}\cup\{+\infty\}$.  Let $\mathbb{R}^n$ denote the set of $n$-dimensional real vectors; $\mathbb{R}^{m\times n}$ denote the set of $m$-by-$n$-dimensional real matrices. 
Let $\|\cdot\|$ denote the Euclidean norm, and let $\partial f(x)$ denote the subdifferential set of $f$ at $x$. Let $\mathcal{N}(x;X)\subset\mathbb{R}^n$ denote the normal cone at point $x\in\mathbb{R}^n$ to convex set $X\subset\mathbb{R}^n$, i.e., $\mathcal{N}(x;X):=\{y\in\mathbb{R}^n:y^T(\tilde{x} - x) \leq 0\ \text{ for all } \tilde{x}\in X\}$. We refer to the convex hull of set $X$ as $\textbf{co}(X)$.

\subsection{Organization}
The rest of this paper is organized as follows. Our problem of interest and proposed algorithm framework are introduced in Section~\ref{sec.problems}. In Section~\ref{sec.perf_guarantee}, we present the our algorithm's convergence properties with different step-size policies and problem settings. The empirical performance of our algorithm is discussed in Section~\ref{sec.num_results}.

\section{Two-Stage Stochastic Linear Programming and Inexact Regularized L-shaped Algorithm}\label{sec.problems}
We consider a two-stage stochastic linear programming problem with fixed recourse (two-stage SLP):
\begin{equation}\label{prob.first_stage}
\begin{aligned}
\min_{x\in \mathbb{R}^n} \ & f(x):=c^{T}x+\EE_{\xi}[Q(x;\xi)] \\
\text{s.t.}\ &Ax = b,\  x\geq 0,
\end{aligned}
\end{equation}
where
\begin{equation}\label{prob.second_stage}
\begin{aligned}
Q(x;\xi)=\min_{y\in\RR^{l}} \left\{q(\xi)^Ty: T(\xi)x+Wy=h(\xi),\ y\geq 0\right\},
\end{aligned} 
\end{equation}
$A\in\RR^{m\times n}$, $b\in\RR^{m}$, $c\in\RR^{n}$, $W\in\RR^{r\times l}$, $\EE_{\xi}[\cdot]$ represents expectation taken with respect to the random variable $\xi$, and $(T(\xi),q(\xi),h(\xi))\in \RR^{r\times n}\times\RR^{l}\times\RR^{r}$ for all $\xi$. For any realization $\xi$, $Q(x;\xi)$ can be obtained by solving a linear program over variable $y$, and we call it the second-stage problem. The outer minimization over the variable $x$ is called the first-stage problem or the master problem. We denote $X\subseteq\mathbb{R}^n$ as the feasible region of the master problem. 

In the second-stage problem~\eqref{prob.second_stage}, $W$ does not depend on the realization of $\xi$, which is called the fixed recourse condition for two-stage SLP. The fixed recourse property is a natural condition to make sure the stability and well-posedness of the second-stage problem and it is well-satisfied in practical problems~(\cite{ShapDentRusz21,OlivSagaSche11}). 
This paper studies the two-tage SLP under the fixed recourse condition to avoid additional notational burdens, and we believe our results can be extended to the general setting where $W$ depends on the realization of $\xi$, as long as the second-stage problem is well-defined and stable.

The L-Shaped algorithm is a classic algorithm for solving two-stage SLP (\cite{Birg88}). It can be viewed as a Benders decomposition of the two-stage stochastic linear programming~\eqref{prob.first_stage}. We here consider an inexact regularized variant of the L-shaped method, which is perhaps the most practical variant of the algorithm (\cite{BielJoha22}). 
In particular, the inexactness provides our algorithm with the ability to solve large-scale optimization problems, while regularization is a common technique that prevents the algorithm from generating unbounded iterates. The method can also be viewed as an inexact proximal bundle method in non-smooth optimization (\cite{OlivSagaSche11}).
To formally introduce the algorithm, we start by providing a (sub)gradient estimation of the function $f(x)$. Notice that the dual of the second-stage problem is,
\begin{align*}
\max_{u\in\RR^{r}} \ & u^{T}(h(\xi)-T(\xi)x)\\
\text{s.t.} \ & W(\xi)^{T}u\leq q(\xi).
\end{align*}
For any $U(x;\xi) \in \arg\max_{u\in\RR^{r}} \left\{u^{T}(h(\xi)-T(\xi)x): W(\xi)^{T}u\leq q(\xi)\right\}$, it follows from the strong duality conditions on linear programming problems and the chain rule that
\begin{equation}\label{eq.subgradient}
-T(\xi)^TU(x;\xi) \in \partial Q(x;\xi)\ .
\end{equation} 

Then it follows \eqref{eq.subgradient} and the subdifferential sum rule (\cite{Clar90}) that
\begin{equation*}
g(x) := c - \mathbb{E}_{\xi}[T(\xi)^TU(x;\xi)] \in\partial f(x).
\end{equation*}

Algorithm~\ref{alg.main} formally presents the inexact regularized L-shaped method. It is a two-loop algorithm. The algorithm is initialized with a solution $x_{0,0}\in X$, a step-size sequence $\{\rho_k\}$ and a descent parameter $\beta\in (0,1)$. At the beginning of each outer iteration $k$, a set of samples $S_k$ is generated from the distribution of $\xi$. Using the sample set $S_k$, we can obtain the function estimation and the (sub)gradient estimation for the $k$-th outer iteration:
\begin{equation}\label{eq:estimation}
    \hf_{k}(x)=c^{T}x+\frac{1}{|S_k|}\sum_{i\in S_k}Q(x;\xi_i) \ , \hg_{k}(x)=c- \frac{1}{|S_k|}\sum_{i\in S_k} T(\xi_i)^T U(x;\xi_i) \in \partial \hf_k(x)\ .
\end{equation}
We can then create the initial bundle function $f_{k,0}:\mathbb{R}^n\to\RR$ as
\begin{equation}\label{eq:initialization}
    f_{k,0}(x)= \hf_k(\xk{0})+\hg_{k}(\xk{0})^T (x-\xk{0}).
\end{equation} 
$f_{k,0}(x)$ provides a lower bound approximation of $\hf_k(x)$ due to the convexity of $\hf_k$. Next, we run the inner loop of the algorithm, which updates the current solution $x_{k,t}$ and the bundle function $f_{k,t}(x)$. More formally, at the $k$-th outer iteration and $t$-th inner iteration, we have a lower bound approximation (bundle) function $f_{k,t}(x)$. We then compute the next inner iterate $\xk{t+1}$ by minimizing a proximal step
\begin{equation}\label{eq.proximal_bundle}
\xk{t+1}=\arg\min_{x\in X} \left\{\fk{t}(x)+\frac{\rho_k}{2}\|x-\xk{0}\|^2\right\},
\end{equation}
where $\rho_k>0$ is the step-size for the $k$-th outer iteration. In our theoretical analysis later in Section \ref{sec.perf_guarantee}, we will discuss the implication of different step-size options.  

Next, the algorithm checks whether there is a sufficient decay at the updated inner iterate $\xk{t+1}$ from \eqref{eq.proximal_bundle} in terms of the noisy function $\hf_k$, i.e.,
\begin{equation}\label{eq.suff_decay}
    \beta(\hf_{k}(\xk{0})-\fk{t}(\xk{t+1})) \leq \hf_{k}(\xk{0})-\hf_{k}(\xk{t+1}).
\end{equation}
If \eqref{eq.suff_decay} is satisfied,  we call it a {\sl serious step}, terminate the inner loop, and go to the next outer iteration with initial solution $x_{k+1,0}\leftarrow\xk{t+1}$. Otherwise, we call it a {\sl null step}, and update the bundle function $f_{k, t+1}(x)$. There are different options for updating the bundle function $f_{k, t+1}(x)$.
The traditional bundle method (\cite{Kell60,Lema78}) utilizes the past sub-gradient information to update the bundle function $f_{k,t+1}(x)$, i.e.,
\begin{equation}\label{eq:generate-full-bundle}
    \fk{t+1}(x)=\max_{j\in J_{k,t+1}\subseteq \{0,\ldots,t+1\}} \left\{\hf_k(\xk{j}) + \hg_k(\xk{j})^T (x- \xk{j})\right\}\ ,
\end{equation}
where $J_{k,t+1}$ is an index set that indicates which past inner iterations are included for generating the bundle function. In particular, if $J_{k,t+1}= \{0,...,t+1\}$ for any $k\geq 0$, \eqref{eq:generate-full-bundle} recovers the bundle function with full memory. Since the full-memory variant of the algorithm keeps track of all previous 
linear lower bound functions, which can be expensive to store when $t$ is large, people in practice usually utilize a limited-memory variant of methods where  $J_{k,t+1} = \{\max\{0,t-m_b+2\},\ldots,t+1\}$ with a predetermined memory size $m_b\geq 1$. Another typical cut used in proximal bundle method is $$\fk{t}(\xk{t+1})+ \sk{t+1}^T(x-\xk{t+1})\ ,$$ where $s_{k,t+1}:= \rho_k(x_{k,0} - x_{k,t+1})$. Notice that $s_{k,t+1} \in \partial f_{k,t}(x_{k,t+1}) + \mathcal{N}(\xk{t+1};X)$ by \eqref{eq.proximal_bundle}, thus 
\begin{equation}\label{eq:func_est_lb_sec2}
\fk{t}(\xk{t+1})+ \sk{t+1}^T(x-\xk{t+1})\le f_{k,t}(x) \le \hf_k(x)
\end{equation}
provides a valid lower bound of the function estimation $\hf_k(x)$ given that $f_{k,t}(x)$ is a lower bound of $\hf_k(x)$. Thus, we can utilize the following class of bundle function, which provides a valid lower bound for $\hf_k(x)$:
\begin{equation}\label{eq:limited-memory}
\begin{aligned}
\fk{t+1}(x)=\max \Big\{&\max_{j\in J_{k,t+1}^g\subseteq \{0,\ldots,t+1\}} \left\{\hf_k(\xk{j}) + \hg_{k}(\xk{j})^T(x- \xk{j}) \right\}, \\
&\max_{j\in J_{k,t+1}^s\subseteq \{1,\ldots,t+1\}} \left\{\fk{j-1}(\xk{j})+ \sk{j}^T(x-\xk{j}) \right\} \Big\} \ .
\end{aligned}
\end{equation}
We can control the memory size of the bundle function by restricting the cardinality of sets $J_{k,t+1}^g$ and $J_{k,t+1}^s$. In particular, our theoretical results only require the most recent cuts to be included in constructing $f_{k,t+1}(x)$, i.e.,  $t+1 \in J_{k,t+1}^g$ and $t+1 \in J_{k,t+1}^s$ (see Assumption \ref{ass:bundle_func} and the discussion therein for more details).

\begin{algorithm}[tb]
	\SetAlgoLined
	Initialize with $x_{0,0}\in X$, a step-size sequence $\{\rho_k\}$, and descent parameter $\beta\in(0,1)$.\\
	\For{outer iteration $k\geq 0$}{
	    Generate a set of samples $S_k$ from the distribution of $\xi$. \\
            Generate function and subgradient estimations $\hf_{k}$ and $\hg_{k}$ based on $S_k$ using \eqref{eq:estimation}. \\
	    Initialize with an approximation function $f_{k,0}(x)= \hf_k(\xk{0})+\hg_{k}(\xk{0})^T (x-\xk{0})$.\\
	    \For{inner iteration $t\geq 0$}{
	        $\xk{t+1}=\arg\min_{x\in X} \{\fk{t}(x)+\frac{\rho_k}{2}\|x-\xk{0}\|^2$\} \\
	        \uIf{$\beta(\hf_{k}(\xk{0})-\fk{t}(\xk{t+1}))\le\hf_{k}(\xk{0})-\hf_{k}(\xk{t+1})$}{
	        Set inner loop length $T_k \leftarrow t+1$\\
	        Obtain a serious step and start a new outer loop with $x_{k+1,0} \leftarrow \xk{t+1}$\\
	        \textbf{break}  \\
	    }
	    \Else{Take a null step and update the approximation function $f_{k,t+1}(x)$ following \eqref{eq:limited-memory}}
		}
	}
	\caption{Inexact Regularized L-shaped Algorithm}
	\label{alg.main}
\end{algorithm}

\section{Computational Guarantees}\label{sec.perf_guarantee}
In this section, we present the non-asymptotic iteration and sample complexity results of Algorithm~\ref{alg.main} under various settings. We start by presenting the assumptions that are needed in our theoretical analysis (Section \ref{sec:ass}), and then present the computational guarantees for Algorithm~\ref{alg.main} with a constant step-size policy (Section \ref{sec.const_step_policy}). After that, we present a superior computational guarantee of the algorithm with an evolving step-size policy (which we dub optimal step-size) motivated by Polyak's step-size in Section \ref{sec:ideal_ss}. This step-size policy achieves the optimal complexity, matching with the information-theoretical lower bound of non-smooth optimization. However, it requires knowing the optimal function value $\min_{x\in X}f(x)$ in the step-size choice. To avoid this, we present a practical heuristic that estimates a lower bound of $\min_{x\in X}f(x)$ over time. Furthermore, if the support of $\xi$ is finite, the optimization problem \eqref{prob.first_stage} satisfies a certain sharpness condition. The complexity results from both the constant step-size policy and the optimal step-size policy can be further improved under the sharpness condition (see Section \ref{sec.conv_sharpness}). Finally, we present the sample complexity results of the optimal step-size policy in Section \ref{sec.sample_const}, i.e, how many samples the algorithm needs to access to identify a solution with a certain accuracy. Actually, with a similar analysis as that in Section \ref{sec.sample_const}, we can also obtain a similar sample complexity result for the other settings.

\subsection{Assumptions}\label{sec:ass}
We introduce the assumptions related to our problem setting and algorithm framework that are used in our analysis. 

Assumption \ref{ass:fixed_recourse} states that the constraint set for the master problem is bounded and non-empty.

\begin{ass}\label{ass:fixed_recourse}
The convex feasible set $X$ is non-empty and bounded with diameter $D$, i.e., $X\neq \varnothing$ and $\|x-\bar{x}\| \leq D$ for all $\{x,\bar{x}\}\subset X$.
\end{ass}

The next assumption is on the estimation error of the inexact function estimation $\hat{f_k}$:

\begin{ass}\label{ass:noise_fg}
There is $\epsilon_1$ and $\epsilon_2$, such that for any outer iteration counter $k$, function estimation $\hat{f}_k$ for all $x\in X$ that
\begin{align*}
\hf_k(x) & \in[f(x)-\epsilon_{1},f(x)+\epsilon_{2}]\ .
\end{align*}
\end{ass}
For any fixed $k$, the existence of the bound $\epsilon_1$ and $\epsilon_2$ are guaranteed since the constraint set $X$ is closed and bounded, and functions $f(x)$ and $f_k(x)$ are continuous. Assumption \ref{ass:noise_fg} assumes a uniform bound over outer iteration $k$. We define
\begin{equation}\label{eq:beps}
\beps = (\beta + 1)(\epsilon_1 + \epsilon_2),
\end{equation}
which will appear in the computational guarantees. Without loss of generality, we suppose $\beps \in (0,1)$. Furthermore, by the construction of $\hf_k$ and $\hg_k$, we have $\hg_k(x) \in\partial \hf_k(x)$ for any $x\in X$, thus it holds by the convexity of $\hf_k(x)$ that
\begin{equation}\label{eq:inexact_func_grad_conv}
\hf_k(\tilde{x})\ge \hf_k(x)+\hg_k(x)^T(\tilde{x}-x)\text{ for all }\tilde{x}\in X\ .
\end{equation}
Assumption \ref{ass:bound_hg} assumes that there exists a uniform bound on the subgradient $\hg_k(x)$, which is a standard assumption in non-smooth optimization:

\begin{ass}\label{ass:bound_hg}
There exists $G>0$ such that for any outer iteration count $k$ and a feasible variable $x\in X$ that $\|\hg_k(x)\| \leq G$.
\end{ass}

We also comment that if the random variable $\xi$ is chosen from a distribution with finite support, which is often the case in practical applications, Assumption \ref{ass:noise_fg} and Assumption \ref{ass:bound_hg} automatically satisfy for our construction of $\hf_k(x)$ and $\hg_k(x)$. This is because $\hf_k(x)$ is a piece-wise linear function of a finite number of choices on a bounded and closed set.

Our next assumption imposes requirements on the construction of the bundle function defined in \eqref{eq:limited-memory}. It essentially requires the two most recent cuts to be included in constructing $f_{k,t+1}(x)$.

\begin{ass}\label{ass:bundle_func}
For any $k\geq 0$ and $0 \leq t\leq T_k-2$, the bundle function $\fk{t+1}:\mathbb{R}^n\to\mathbb{R}$ defined in \eqref{eq:limited-memory} is always convex, $G$-Lipschitz continuous, and satisfies that for any $x\in\mathbb{R}^n$
\begin{equation*}
\hf_k(x) \geq \fk{t+1}(x) \geq \max\left\{ \hf_k(\xk{t+1}) + \hg_{k}(\xk{t+1})^T(x- \xk{t+1}), \fk{t}(\xk{t+1})+ \sk{t+1}^T(x-\xk{t+1}) \right\},
\end{equation*}
where $\sk{t+1} = \rho_k(\xk{0} - \xk{t+1})$ and $G > 0$ is defined in Assumption~\ref{ass:bound_hg}.
\end{ass}

We note that the classic full-memory bundle function (see \eqref{eq:generate-full-bundle} with $J_{k,t+1} = \{0,\ldots,t+1\}$ for any $k\geq 0$) satisfies Assumption~\ref{ass:bundle_func}. For the classic full-memory bundle function, the upper bound in Assumption~\ref{ass:bundle_func} follows from the convexity of $\hf_k$. Meanwhile, the lower bound in Assumption~\ref{ass:bundle_func} follows from the definition of $f_{k,t+1}$ and \eqref{eq:func_est_lb_sec2}.

\subsection{Constant Step-size Policy}\label{sec.const_step_policy}
A constant step-size policy refers to the step-size used in Algorithm \ref{alg.main} for different outer iterations $k$ being the same, i.e., $\rho_k = \rho$ for all $k\geq 0$. It is perhaps the most basic and useful step-size selection strategy for solving two-stage SLP. Due to the nature of inexact estimates (see Assumption~\ref{ass:noise_fg}), one could at best expect Algorithm~\ref{alg.main} to generate a sequence of iterates whose objective values converge to a neighborhood of the optimal objective function value. Theorem \ref{thm:const_main} specifies this radius of the neighborhood as well as the total number of outer iterations and inner iterations of Algorithm~\ref{alg.main} to reach this neighborhood in terms of the problem and algorithm parameters.

\begin{thm}\label{thm:const_main}
Suppose Assumptions~\ref{ass:fixed_recourse}--\ref{ass:bundle_func} hold. Consider Algorithm \ref{alg.main} using a constant step-size policy, i.e., $\rho_k=\rho>0$ for all outer iterations $k\geq 0$. Let's define 
\begin{equation}\label{eq:const_neighborhood}
\delta^C := \max\bracket{\frac{4\beps}{\beta}, \sqrt{\frac{4\beps \rho}{\beta} }D} \quad \text{and} \quad K^C := \max\left\{0,\left\lceil\frac{f(x_{0,0}) - f^* - \delta^C}{\bar\epsilon}\right\rceil\right\},
\end{equation}
where $\beps$ is defined in \eqref{eq:beps} and $f^*$ is the optimal function value of $f:X\to\mathbb{R}$, then there exists an outer iteration $k \leq K^C$ such that $f(\xk{0}) - f^* \leq \delta^C$. Furthermore, the total number of inner steps (including null steps) before such a solution is found is at most 
\begin{equation}\label{eq:total-null-const}
\pran{\left\lceil{\frac{8G^2}{\rho(1-\beta)^2\beps} - \frac{16}{(1-\beta)^2}}\right\rceil + 1}K^C. 
\end{equation}
\end{thm}

As a direct consequence of Theorem \ref{thm:const_main}, if we choose constant step sizes $\rho_k = \rho = \frac{4\beps}{\beta D^2}$ for all outer iterations $k\geq 0$,  Algorithm~\ref{alg.main} needs $\mathcal{O}\left(\frac{1}{\beps^3}\right)$ number of inner iterations to find an iterate whose objective gap is no larger than $\mathcal{O}(\beps)$.

The rest of this section proves Theorem~\ref{thm:const_main}. We first introduce the proximal gap at the $k$-th outer iteration as
\begin{equation}\label{eq.Delta_k}
\Delta_k=f(\xk{0})-\pran{f(\bx_{k+1})+\frac{\rho_k}{2}\|\bx_{k+1}-\xk{0}\|^2} \text{ with } \bx_{k+1}=\arg\min_{x\in X}\left\{ f(x)+\frac{\rho_k}{2}\|x-\xk{0}\|^2\right\}, 
\end{equation}
where $\rho_k>0$ represents the step size at outer iteration $k$. This proximal gap term is a valid optimality measure of $\xk{0}$; see \cite{ruszczynski2011nonlinear}. It follows from the convexity of $f(x)$ that (i) $\Delta_k \geq 0$ and (ii) $\Delta_k = 0$ if and only if $\xk{0} = \arg\min_{x\in X} f(x)$. 

Next, we present two lemmas that we will use in the proof of Theorem \ref{thm:const_main}. The proofs of the two lemmas are presented in Appendix \ref{app:proof_const}. The first lemma shows that $\Delta_k$ always performs as an upper bound onquantities related to the optimal function value gap.

\begin{lem}\label{lem:proximal_bound}
Suppose Assumption~\ref{ass:fixed_recourse} holds, then for any outer iteration $k\geq 0$, with step size $\rho_k>0$ from Algorithm~\ref{alg.main} and the diameter $D$ defined in Assumption~\ref{ass:fixed_recourse}, we have that 
\begin{equation}\label{eq:proximal_bound}
    \Delta_k \ge \left\{ \begin{array}{cl}
        (f(\xk{0})-f^*)^2/(2\rho_kD^2) & \text{ if } f(\xk{0})-f^* \le \rho_k D^2, \\
        (f(\xk{0})-f^*)/2 & \text{ otherwise, } 
    \end{array} \right.
\end{equation}
where $f^*$ represents the optimal function value of $f:X\to\RR$ defined in \eqref{prob.first_stage}.
\end{lem}

Since the proximal gap $\Delta_k$ cannot be computed by Algorithm~\ref{alg.main}, which relies on the true function $f(x)$, we introduce another computable inexact proximal gap term. For any $k\geq 0$ and $t\in \{0,\ldots,T_k-1\}$, we define the inexact proximal gap $\tDeltak{t}$ as
\begin{equation}\label{eq.inner_prox_gap}
\tDeltak{t} = \hf_k(\xk{0})-\pran{\fk{t}(\xk{t+1})+\frac{\rho_{k}}{2}\|\xk{t+1}-\xk{0}\|^2},
\end{equation}
then it follows from Algorithm~\ref{alg.main} that every null step satisfies $\beta(\hf_{k}(\xk{0})-\fk{t}(\xk{t+1})) > \hf_{k}(\xk{0})-\hf_{k}(\xk{t+1})$, thus for any $k \geq 0$ and $0 \leq t \leq T_k - 2$
\begin{align}\label{eq:null-step}
    \hf_{k}(\xk{t+1})-\fk{t}(\xk{t+1}) > (1-\beta)\pran{\hf_{k}(\xk{0})-\fk{t}(\xk{t+1})}\ge (1-\beta)\tDelta_{k,t}.
\end{align}
Lemma~\ref{lem:Deltakt_connection} builds up connections between any two consecutive inexact proximal gap terms, e.g., $\tDeltak{t-1}$ and $\tDeltak{t}$.
\begin{lem}\label{lem:Deltakt_connection}
Under the same assumption as Theorem \ref{thm:const_main}, we have $\|s_{k,t}\| := \|\rho_k(\xk{0} - \xk{t})\| \leq G$ for any $k\geq 0$ and $1 \leq t \leq T_k$, where $G$ is defined in Assumption~\ref{ass:bound_hg}. Moreover, it holds for any $k\geq 0$ and $1 \leq t\le T_k-1$ that
\begin{equation}\label{eq:combined_both}
\tDeltak{t} \leq \tDeltak{t-1} -\theta_{k,t}(\hf_k(\xk{t})-\fk{t-1}(\xk{t}))+\frac{\theta_{k,t}^2}{2\rho_k}\|\sk{t}-\hg_{k}(\xk{t})\|^2,
\end{equation}
where $\theta_{k,t}=\min\bracket{1,\frac{\rho_k(\hf_k(\xk{t})-\fk{t-1}(\xk{t}))}{\|\hg_{k}(\xk{t})-\sk{t}\|^2}}$.
\end{lem}
 Now we are ready to prove Theorem \ref{thm:const_main}. The proof of Theorem~\ref{thm:const_main} can be divided into three steps. First, based on the result of Lemma~\ref{lem:Deltakt_connection}, we present an upper bound of the length on the inner iterations for an outer iteration. Next, we present a sufficient decrease property on the objective function values between two consecutive {\sl serious} iterates. Last, we combine the first two components together, which leads to the complexity theory of the algorithm.

\begin{proof}[Proof of Theorem~\ref{thm:const_main}]
This proof includes three parts as follows.

\textbf{Part 1 (Length of the inner loop)} \quad Without loss of generality, let's suppose current outer iteration $k\geq 0$ satisfies $f(\xk{0}) - f^* > \delta^C$. Then from Lemma~\ref{lem:proximal_bound} and the definitions of $\delta^C$ and $\beps$, we know $\Delta_k > \epsilon_1 + \epsilon_2$. Based on the result of Lemma~\ref{lem:Deltakt_connection}, we further consider two cases.
\begin{itemize}
\item If $\frac{\rho_k(\hf_k(\xk{t})-\fk{t-1}(\xk{t}))}{\|\hg_{k}(\xk{t})-\sk{t}\|^2} \geq 1$, namely $\theta_{k,t}=1$, then \eqref{eq:combined_both} implies that
\begin{equation}\label{eq:case-1}
\begin{aligned}
    \tDeltak{t} &\le \tDeltak{t-1} -(\hf_k(\xk{t})-\fk{t-1}(\xk{t})) + \frac{1}{2\rho_k}\|\sk{t}-\hg_{k}(\xk{t})\|^2 \\
    &\le \tDeltak{t-1} - \frac{1}{2}(\hf_k(\xk{t})-\fk{t-1}(\xk{t})) \le \tDeltak{t-1} -\frac{1}{2}(1-\beta)\tDeltak{t-1},
\end{aligned}
\end{equation}
where the last inequality follows \eqref{eq:null-step} since we take a null step at inner iteration $t-1$. 
\item If $\frac{\rho_k(\hf_k(\xk{t})-\fk{t-1}(\xk{t}))}{\|\hg_{k}(\xk{t})-\sk{t}\|^2} < 1$, then it follows Lemma~\ref{lem:Deltakt_connection} that $\theta_{k,t}=\frac{\rho_k(\hf_k(\xk{t})-\fk{t-1}(\xk{t}))}{\|\hg_{k}(\xk{t})-\sk{t}\|^2}$. Thus, by Assumption~\ref{ass:bound_hg}, Lemma~\ref{lem:Deltakt_connection}, \eqref{eq:null-step}, and \eqref{eq:combined_both}, we have
\begin{equation}\label{eq:case-2}
\begin{aligned}
\tDeltak{t} & \leq \tDeltak{t-1} -\theta_{k,t}(\hf_k(\xk{t})-\fk{t-1}(\xk{t}))+\frac{\theta_{k,t}^2}{2\rho_k}\|\sk{t}-\hg_{k}(\xk{t})\|^2 \\ &= \tDeltak{t-1} - \frac{\rho_k(\hf_k(\xk{t})-\fk{t-1}(\xk{t}))^2}{2\|\hg_{k}(\xk{t})-\sk{t}\|^2} \leq \tDeltak{t-1} - \frac{\rho_k(1-\beta)^2\tDeltak{t-1}^2}{2\|\hg_{k}(\xk{t})-\sk{t}\|^2} \\
& \leq \tDeltak{t-1} - \frac{\rho_k(1-\beta)^2\tDeltak{t-1}^2}{4(\|\hg_{k}(\xk{t})\|^2 + \|\sk{t}\|^2)} \leq \tDeltak{t-1} - \frac{\rho_k(1-\beta)^2\tDeltak{t-1}^2}{8G^2}.
\end{aligned}
\end{equation}
where the last inequality follows the fact of $\max\{\|\hg_{k}(\xk{t})\|,\|s_{k,t}\|\} \leq G$ (see Assumption~\ref{ass:bound_hg} and Lemma~\ref{lem:Deltakt_connection}). 
\end{itemize}

By \eqref{eq:case-1} and~\eqref{eq:case-2}, $\tDeltak{t}$ monotonically decreases when $t$ increases.
From Assumption~\ref{ass:bound_hg}, \eqref{eq:initialization}, and \eqref{eq.inner_prox_gap}, we further have that 
\begin{equation}\label{eq:bound-delta0}
\begin{aligned}
\tDeltak{t-1} \leq \tDelta_{k,0} &=\hf_k(\xk{0})-\pran{f_{k,0}(\xk{1})+\frac{\rho_{k}}{2}\|\xk{1}-\xk{0}\|^2} \\
&= \hg_k(\xk{0})^T(\xk{0} - \xk{1}) - \frac{\rho_{k}}{2}\|\xk{1}-\xk{0}\|^2 \\
&\leq G\|\xk{1}-\xk{0}\| - \frac{\rho_{k}}{2}\|\xk{1}-\xk{0}\|^2 \leq \frac{G^2}{2\rho_k},
\end{aligned}
\end{equation}
while the last inequality is from the optimal value of a strongly concave quadratic optimization problem $\max_{x\in\mathbb{R}}\{Gx - \frac{\rho_k}{2}x^2\}$.
Combining \eqref{eq:case-1},~\eqref{eq:case-2} and~\eqref{eq:bound-delta0}, we obtain
\begin{align}\label{eq:decay}
    \tDeltak{t} \le \tDeltak{t-1}-\min\bracket{\frac{(1-\beta)}{2}\tDeltak{t-1}, \frac{\rho_k(1-\beta)^2\tDeltak{t-1}^2}{8G^2}} = \tDeltak{t-1} - \frac{\rho_k(1-\beta)^2\tDeltak{t-1}^2}{8G^2}.
\end{align}
Once a null step is taken at iteration $(t-1)$, from \eqref{eq:decay} and the fact of $\tDeltak{t}\leq \tDeltak{t-1}$, we have
\begin{align*}
\tDeltak{t}\leq \tDeltak{t-1} - \frac{\rho_k(1-\beta)^2\tDeltak{t-1}^2}{8G^2} \leq \tDeltak{t-1} - \frac{\rho_k(1-\beta)^2\tDeltak{t-1}\tDeltak{t}}{ 8G^2},
\end{align*}
which implies that 
\begin{align}\label{eq:decay_2}
\frac{1}{\tDeltak{t}} \geq \frac{1}{\tDeltak{t-1}} + \frac{\rho_k(1-\beta)^2}{8G^2}.
\end{align}
From \eqref{eq:initialization} and Assumptions~\ref{ass:noise_fg} and~\ref{ass:bundle_func}, we have that for all $x\in X$, $k\geq 0$ and $0\leq t\leq T_k - 1$
\begin{equation}\label{eq.hf_sandwich}
\fk{t}(x) \le \hf_k(x) \leq f(x) + \epsilon_2.
\end{equation} 
Therefore, by Assumption~\ref{ass:noise_fg}, \eqref{eq.Delta_k}, \eqref{eq.inner_prox_gap}, \eqref{eq.hf_sandwich} and the definition of $\xk{T_k}$, we have that
\begin{equation}\label{eq:error}
\begin{aligned}
    \tDeltak{T_k-1} &= \hf_k(\xk{0})-\pran{\fk{T_k-1}(\xk{T_k})+\frac{\rho_{k}}{2}\|\xk{T_k}-\xk{0}\|^2} \\
    &= \hf_k(\xk{0})-\min_{x\in X}\left\{\fk{T_k-1}(x) + \frac{\rho_{k}}{2}\|x-\xk{0}\|^2\right\} \\
    &\geq (f(\xk{0}) - \epsilon_1) - \left(\fk{T_k-1}(\bar{x}_{k+1}) + \frac{\rho_{k}}{2}\|\bar{x}_{k+1} - \xk{0}\|^2\right) \\
    &\geq  (f(\xk{0}) - \epsilon_1) - \left(f(\bar{x}_{k+1}) + \epsilon_2 + \frac{\rho_{k}}{2}\|\bar{x}_{k+1} - \xk{0}\|^2\right) \\
    &= \Delta_{k}-\epsilon_1-\epsilon_2.
\end{aligned}
\end{equation}
Recall that in every outer iteration $k\geq 0$, we take null steps for inner iterations $t\in\{0,...,T_k-2\}$. Combining \eqref{eq:bound-delta0},~\eqref{eq:decay_2}, and~\eqref{eq:error}, we know the number of null steps could be upper bounded by 
\begin{equation}\label{eq.Tk_bound}
\begin{aligned}
T_k - 1 \leq \left\lceil{\frac{\frac{1}{\Delta_k - \epsilon_1 - \epsilon_2} - \frac{2\rho_k}{G^2}}{\frac{\rho_k(1-\beta)^2}{8G^2}}}\right\rceil \leq \left\lceil{\frac{8G^2}{\rho_k(1-\beta)^2(\Delta_k - \epsilon_1 - \epsilon_2)} - \frac{16}{(1-\beta)^2}}\right\rceil,
\end{aligned}
\end{equation}
which provides an upper bound for $T_k$.

\textbf{Part 2 (Objective decrease on serious iterates)} \quad Now we are going to show that there is an objective function decrease on two consecutive serious iterates. It follows Assumption~\ref{ass:noise_fg} and the definition of the serious step that
\begin{equation}\label{eq.obj_decay}
\begin{aligned}
    f(x_{k+1,0})&\le \hf_{k}(x_{k+1,0})+\epsilon_1 = \hf_{k}(x_{k,T_k})+\epsilon_1 \\
    &\le \hf_k(\xk{0}) - \beta(\hf_k(\xk{0})-\fk{T_k-1}(\xk{T_k}))+\epsilon_1\\
    &\le \hf_k(\xk{0}) - \beta\pran{\hf_k(\xk{0})-\pran{\fk{T_k-1}(\xk{T_k})+\frac{\rho_k}{2}\|\xk{T_k}-\xk{0}\|^2}}+\epsilon_1\\
    &\le \hf_k(\xk{0}) - \beta\pran{\hf_k(\xk{0})-\pran{\fk{T_k - 1}(\bx_{k+1})+\frac{\rho_k}{2}\|\bx_{k+1}-\xk{0}\|^2}}+\epsilon_1\\
    &\le \hf_k(\xk{0}) - \beta\pran{(f(\xk{0}) - \epsilon_1)-\pran{f(\bx_{k+1})+\frac{\rho_k}{2}\|\bx_{k+1}-\xk{0}\|^2+\epsilon_2}}+\epsilon_1\\
    & = \hf_k(\xk{0}) - \beta\Delta_k+\beta(\epsilon_1+\epsilon_2)+\epsilon_1 \\
    & \le f(\xk{0}) - \beta\Delta_k +(\beta+1)(\epsilon_1+\epsilon_2) \\
    & = f(\xk{0})-\left(\beta \Delta_k - \beps \right),
\end{aligned}
\end{equation}
where the fourth inequality follows the update rule of $\xk{T_k}$, the fifth inequality is from Assumption~\ref{ass:noise_fg} and \eqref{eq.hf_sandwich}, and the second equality comes from \eqref{eq.Delta_k}.

\textbf{Part 3 (Total outer and inner iterations)} \quad We prove this part by contradiction. Suppose for all outer iterations $0\leq k\le K^C$ it holds that $f(\xk{0})-f^* > \delta^C$. Then we have from \eqref{eq:const_neighborhood} and \eqref{eq:proximal_bound} that
\begin{equation}\label{eq:const-Deltak}
\begin{aligned}
    \Delta_k &\ge \min\bracket{\frac{(f(\xk{0})-f^*)^2}{2\rho D^2}, \frac{f(\xk{0})-f^*}{2}} > \min\bracket{\frac{(\delta^C)^2}{2\rho D^2}, \frac{\delta^C}{2}} \ge \frac{2\beps}{\beta}.
\end{aligned}
\end{equation}
Thus, it follows from \eqref{eq.obj_decay} that
$$
f(x_{k+1,0})\le f(\xk{0})- \left(\beta\Delta_k - \beps \right) < f(\xk{0})- \beps .
$$
As a result, we have $f(x_{K^C,0})-f^*< f(x_{0,0})-f^*-K^C\beps \leq \delta^C$, which results in a contradiction. Therefore, there must exist  some $0\leq k\leq K^C$ such that $f(\xk{0})-f^* \leq \delta^C$.

Furthermore, notice that it follows from \eqref{eq:const-Deltak} and the definition of $\bar{\epsilon}$ that $\Delta_k > \frac{2\beps}{\beta} \geq \bar\epsilon + \epsilon_1+\epsilon_2$. By \eqref{eq.Tk_bound}, we obtain the total number of null steps is at most 
\begin{equation*}
\pran{\left\lceil{\frac{8G^2}{\rho(1-\beta)^2\beps} - \frac{16}{(1-\beta)^2}}\right\rceil + 1}K^C, 
\end{equation*}
which completes the proof.
\end{proof}
Let's revisit the three parts in the proof of Theorem~\ref{thm:const_main}. In the first part, by exploring a recursive relationship of consecutive inexact proximal gap terms, we have shown that the number of null steps in a single outer iteration is upper bounded. Then in the second part, we have proved an objective function decrease property on consecutive serious iterates. We note that the proof of the first two parts can be applied to any general step-size policies, not only the constant step-size policy. Finally, based on the results of the first two parts, we have provided upper bounds for the total number of outer and inner iterations under a constant step-size policy, which concludes the purpose of this subsection.

We comment here that our proof technique in Theorem \ref{thm:const_main} follows from~\cite{DiazGrim21}. Theorem \ref{alg.main} can be viewed as an extension of Theorem 2.1 for exact proximal bundle methods~\cite{DiazGrim21} into the inexact setting. Unlike the exact algorithm which can identify a solution with an arbitrarily small optimality gap, the inexact algorithm can at best find a solution that is at the order of the inexactness level $\beps$. Surprisingly, the inexact function and gradient estimation do not significantly slow down the convergence behaviors of the algorithm. The fundamental difficulty of the analysis in the inexact setting comes from determining how to navigate the error terms in the inexact approximation so that the errors do not become exaggerated over time.

\subsection{Optimal Step-Size Policy}\label{sec:ideal_ss}

Theorem \ref{thm:const_main} shows that Algorithm \ref{alg.main} with the constant step-size needs $\mathcal{O}(\frac{1}{\beps^3})$ number of inner iterations to find a solution with $\mathcal{O}\pran{{\beps}}$ optimality gap. In this section, we show that with a carefully chosen step-size motivated by Polyak's step-size in non-smooth optimization, Algorithm \ref{alg.main} can find a solution with $\mathcal{O}\pran{{\beps}}$ optimality gap in $\mathcal{O}(\frac{1}{\beps^2})$ number of inner iterations.  Furthermore, the complexity of the algorithm with such a step-size matches the complexity lower bound of non-smooth optimization; thus, we call it the optimal step-size policy. Similar to Polyak's step size policy (\cite{Poli87}), the optimal step-size requires the knowledge of $f^*$.
In practice, the value of $f^*$ is usually unknown before running the algorithm, which makes the algorithm impractical. However, one may estimate an approximation of $f^*$ while running the algorithm, and we discuss a practical step-size rule at the end of this section. 

More specifically, with the optimal step-size policy for Algorithm \ref{alg.main}, the step-size at the $k$-th outer iteration is chosen as
\begin{equation}\label{eq.ideal_step_rule}
\rho_k=\frac{\hf_k(\xk{0})-\epsilon_2-f^*}{D^2} \ ,
\end{equation}
where $f^*$ is the optimal objective value, and $D$ is the diameter of the constraint set (see Assumption~\ref{ass:fixed_recourse}).
We note that if $\rho_k \leq 0$, we directly terminate the algorithm since Assumption~\ref{ass:noise_fg} directly implies that $\xk{0}$ is already near-optimal, i.e., $f(\xk{0}) \leq f^* + \epsilon_1 + \epsilon_2$. Therefore, for our optimal step-size algorithm, we only consider the case that $\rho_k > 0$ for all $k\geq 0$. Moreover, by Assumption~\ref{ass:noise_fg} and the condition of $\rho_k > 0$, we know 
\begin{equation*}
f(\xk{0})-f^* \ge \hf_k(\xk{0})-\epsilon_2-f^* = \rho_k D^2,
\end{equation*}
and it follows from Lemma \ref{lem:proximal_bound} that 
\begin{equation}\label{eq.Deltak_lb_gap_ideal}
\Delta_k\ge \frac{1}{2}(f(\xk{0})-f^*).
\end{equation}
The next theorem presents the complexity theory of Algorithm~\ref{alg.main} with the optimal step-size policy \eqref{eq.ideal_step_rule}.

\begin{thm}\label{thm:ideal_main}
We consider Algorithm~\ref{alg.main} using the optimal step-size policy $\rho_k=\frac{\hf_k(\xk{0})-\epsilon_2-f^*}{D^2} > 0$ for all $k\geq 0$. Suppose Assumptions~\ref{ass:fixed_recourse}--\ref{ass:bundle_func} hold. Let's define  
\begin{equation}
\delta^I:=\frac{4\beps}{\beta} \quad \text{and} \quad K^I:= \max\left\{0, \left\lceil \frac{\log\pran{\frac{f(x_{0,0})-f^*}{\delta^I}}}{-\log\pran{1-\frac{\beta}{4}}} \right\rceil\right\},
\end{equation}
where $\beps$ is defined in \eqref{eq:beps} and $f^*$ is the optimal function value of $f:X\to\mathbb{R}$; then there exists an outer iteration $k\le K^I$ such that $f(x_{k,0})-f^*\le \delta^I$. Furthermore, the total number of inner steps (including null steps) before such a solution is found is at most
\begin{equation}\label{eq:total-null-ideal}
\frac{32G^2D^2}{3(1-\beta)^2(2-\frac{\beta}{4})\beps^2} = \mathcal{O}\pran{\frac{1}{\beps^2}}. 
\end{equation}
\end{thm}

By comparing Theorems~\ref{thm:const_main} and~\ref{thm:ideal_main}, we see that to find an $\mathcal{O}(\beps)$-optimal solution, Algorithm~\ref{alg.main} with constant step sizes needs $\mathcal{O}(\frac{1}{\beps})$ outer iterations and $\mathcal{O}(\frac{1}{\beps^3})$ inner iterations, while Algorithm~\ref{alg.main} with optimal step sizes only requires $\mathcal{O}(\log(1/\beps))$ outer iterations and $\mathcal{O}(\frac{1}{\beps^2})$ inner iterations. In particular, Theorem \ref{thm:ideal_main} shows that Algorithm \ref{alg.main} with the optimal step-size takes $\mathcal{O}\left(\frac{1}{\beps^2}\right)$ function/gradient evaluations to reach an $\mathcal{O}\left(\beps\right)$-optimal solution, matching the lower bound of first-order gradient-based methods for solving non-smooth convex optimization problems; see \cite[Theorem~7.2.1]{Nemi95}.
This fact effectively demonstrates the advantage of the optimal step-size variant of Algorithm~\ref{alg.main}.

\begin{proof}[Proof of Theorem~\ref{thm:ideal_main}]
The proof follows the same structure as the proof of Theorem \ref{thm:const_main}. The difference is mainly the third step, where we utilize the optimal step-size policy \eqref{eq.ideal_step_rule} to bound the total number of iterations. Without loss of generality, we assume $f(x_{0,0})-f^*>\delta^I$, thus $K^I > 0$.

\textbf{Part 1 (Length of the inner loop)} \quad With the same argument as the first part in the proof of Theorem~\ref{thm:const_main}, we know that \eqref{eq.Tk_bound} holds.

\textbf{Part 2 (Objective decrease on serious iterates)} \quad It follows from the same proof as in the second part of the proof of Theorem~\ref{thm:const_main} that \eqref{eq.obj_decay} holds.

\textbf{Part 3 (Total outer and inner iterations)} \quad
We prove the first statement by contradiction. Suppose for any outer iteration $k\le K^I$, it holds that $f(\xk{0})-f^* > \delta^I=\frac{4\beps}{\beta}$. 
It follows from \eqref{eq.obj_decay} and \eqref{eq.Deltak_lb_gap_ideal} that
\begin{equation}\label{eq:ideal_recur}
\begin{aligned}
f(x_{k+1,0})-f^*&\le f(x_{k,0})- f^* - \left(\beta\Delta_k - \beps \right) \\
&< f(x_{k,0})- f^* - \left(\frac{\beta}{2}(f(x_{k,0})- f^*) - \frac{\beta}{4}(f(x_{k,0})- f^*)\right) \\
&= \pran{1-\frac{\beta}{4}}(f(x_{k,0})- f^*).
\end{aligned}
\end{equation}
Therefore, there is a linear decay in the optimal function value gap after each serious step. Thus,
$$
f(x_{K^I,0})-f^* \le \pran{1-\frac{\beta}{4}}^{K_I}(f(x_{k,0})- f^*)\le \delta^I,
$$
which leads to a contradiction and proves the first part of the theorem.
Next, we consider the total number of inner iterations. Let's denote $K^I_{\max}\geq 0$ as the first outer iteration such that $f(x_{K^I_{\max},0}) - f^* > \delta^I$ and $f(x_{K^I_{\max}+1,0}) - f^* \leq \delta^I$. It follows \eqref{eq.Deltak_lb_gap_ideal} that
\begin{equation}\label{eq:Delta_Kmax_lb}
\Delta_{K^I_{\max}} \ge \frac{1}{2}(f(x_{K^I_{\max},0})-f^*) > \frac{\delta^I}{2} = \frac{2\bar\epsilon}{\beta} \ge \bar\epsilon + \epsilon_1+\epsilon_2,
\end{equation}
where the last equality and the last inequality are from definitions of $\delta^I$ and $\beps$ from the theorem statement. Meanwhile, using \eqref{eq.ideal_step_rule} and Assumption~\ref{ass:noise_fg}, we have
\begin{equation}\label{eq:rho_Kmax_lb}
\rho_{K^I_{\max}} = \frac{\hf_{K^I_{\max}}(x_{K^I_{\max},0})-\epsilon_2-f^*}{D^2} \geq \frac{f(x_{K^I_{\max},0}) - f^* - \epsilon_1 - \epsilon_2}{D^2} > \frac{\delta^I - \epsilon_1 - \epsilon_2}{D^2} > \frac{3\beps}{\beta D^2},
\end{equation}
where the last inequality uses definitions of $\delta^I$ and $\beps$ from the theorem statement. 
Furthermore, from \eqref{eq:ideal_recur} and the same logic as in \eqref{eq:Delta_Kmax_lb} and \eqref{eq:rho_Kmax_lb}, for any $k\in\{0,\ldots,K^I_{\max}\}$,
\begin{equation}\label{eq:lb_Deltak_rhok}
\begin{aligned}
\Delta_{k}& \ge \frac{1}{2}(f(x_{k,0})-f^*) \geq \frac{f(x_{K^I_{\max},0})-f^*}{2(1-\frac{\beta}{4})^{(K^I_{\max}-k)}} > \frac{\delta^I}{2(1-\frac{\beta}{4})^{(K^I_{\max}-k)}} \\
&= \frac{2\bar\epsilon}{\beta(1-\frac{\beta}{4})^{(K^I_{\max}-k)}} \ge \frac{\bar\epsilon}{(1-\frac{\beta}{4})^{(K^I_{\max}-k)}} + \epsilon_1+\epsilon_2, \\
\text{and}\quad \rho_{k} &= \frac{\hf_{k}(x_{k,0})-\epsilon_2-f^*}{D^2} \geq \frac{f(x_{k,0}) - f^* - \epsilon_1 - \epsilon_2}{D^2} \\
&\geq \frac{\frac{f(x_{K^I_{\max},0})-f^*}{(1-\frac{\beta}{4})^{(K^I_{\max}-k)}} - \epsilon_1 - \epsilon_2}{D^2} > \frac{\frac{\delta^I}{(1-\frac{\beta}{4})^{(K^I_{\max}-k)}} - \epsilon_1 - \epsilon_2}{D^2} > \frac{3\beps}{\beta(1-\frac{\beta}{4})^{(K^I_{\max}-k)} D^2}.
\end{aligned}
\end{equation}
Therefore, by \eqref{eq.Tk_bound} and \eqref{eq:lb_Deltak_rhok}, we obtain the total number of null steps is at most 
\begin{align*}
&\sum_{k=0}^{K^I_{\max}} \left(\left\lceil{\frac{8G^2}{\rho_k(1-\beta)^2(\Delta_k - \epsilon_1 - \epsilon_2)} - \frac{16}{(1-\beta)^2}}\right\rceil + 1\right) \\
< \ &\sum_{k=0}^{K^I_{\max}} \left(\left\lceil{\frac{8G^2}{\frac{3\beps}{\beta(1-\frac{\beta}{4})^{(K^I_{\max}-k)} D^2}\cdot(1-\beta)^2\cdot\frac{\bar\epsilon}{(1-\frac{\beta}{4})^{(K^I_{\max}-k)}}} - \frac{16}{(1-\beta)^2}}\right\rceil + 1\right) \\
= \ &\sum_{k=0}^{K^I_{\max}} \left(\left\lceil{\frac{8\beta(1-\frac{\beta}{4})^{2(K^I_{\max}-k)}G^2D^2}{3(1-\beta)^2\beps^2} - \frac{16}{(1-\beta)^2}}\right\rceil + 1\right) \\
\leq \ &\sum_{k=0}^{K^I_{\max}} \left(\frac{8\beta(1-\frac{\beta}{4})^{2(K^I_{\max}-k)}G^2D^2}{3(1-\beta)^2\beps^2}\right) \\
= \ &\frac{8\beta G^2D^2}{3(1-\beta)^2\beps^2}\cdot\frac{1-(1-\frac{\beta}{4})^{2(K^I_{\max}+1)}}{1-(1-\frac{\beta}{4})^2} < \frac{32G^2D^2}{3(1-\beta)^2(2-\frac{\beta}{4})\beps^2},
\end{align*}
which concludes the statement.
\end{proof}

We note that while the optimal step-size policy enjoys a strong theoretical convergence rate, it may not be a practical algorithm, because it requires knowing the optimal objective function value $f^*$. To overcome this issue, an alternative is to use multiple step sizes in parallel. More specifically, one can try out multiple step-size options $\rho = 2^{-\tau}\bar{\rho}$ with $\tau \in \{0,1,\ldots,M\}$ in parallel in the inner loop and use the first serious step found by one of the step-sizes. Following the same argument as that in \cite{ReneGrim22}, we believe one can obtain a similar complexity result of Theorem \ref{thm:ideal_main} up to an additional $\log$ term. However, to reach a reasonable tolerance, the value of $M$ is usually chosen around or larger than $20$, which means that the computational cost for one iteration is $20$ times more than that for a standard step. Thus, while the parallel step-size policy can have theoretical guarantees, it is not a practical algorithm.
We propose here a practical step-size policy that is motivated by the optimal step-size policy, where we use the current bundle function value to approximate $f^*$. More specifically, we set
\begin{equation}\label{eq:practical_stepsize_policy}
\rho_{k} = \begin{cases} 
C_P (\hf_k(\xk{0}) - f_{k-1,T_{k-1}-1}(\xk{0})) & \text{if } k\neq 0\text{ and }\hf_k(\xk{0}) > f_{k-1,T_{k-1}-1}(\xk{0}), \\
C_P & \text{otherwise,} \end{cases}
\end{equation}
where $C_P$ is a hyperparameter of the algorithm. 
One may see from \eqref{eq:practical_stepsize_policy} that we try to use $f_{k-1,T_{k-1}-1}(\xk{0}) = f_{k-1,T_{k-1}-1}(x_{k-1,T_{k-1}})$ to approximate $f^*$. If $\hf_k(\xk{0}) \le f_{k-1,T_{k-1}-1}(\xk{0})$, i.e., this specification does not lead to a valid (positive) step-size choice, then this result triggers a safeguard step-size by setting $\rho_{k} = C_P > 0$. While we do not have theoretical support for this algorithm, the numerical study in Section~\ref{sec.num_results} demonstrates this proposed method's superior behavior. 

\subsection{Complexity under the Sharpness Condition}\label{sec.conv_sharpness}

Sharpness is a well-studied condition in convex optimization and is defined below:

\begin{mydef}\label{def:sharpness}
A convex function $f:X\to\mathbb{R}$ is sharp with parameter $\mu>0$ if $$f(x) - f^* \geq \mu \cdot {\rm dist}(x,X^*)$$ for all $x\in X$, where $X^*$ is the optimal solution set of the function $f$, and ${\rm dist}(x,X^*) = \min_{x^*\in X^*}\|x - x^*\|$ is the distance between $x$ and $X^*$.
\end{mydef}

The sharpness condition characterizes the linear growth of the objective function values around the optimal solution set. One can easily see that a piecewise linear convex function is a sharp function. Notice that the objective $f(x)$ in a two-stage SLP with a finite support of $\xi$ is a piecewise linear function, thus it is a sharp function. In this subsection, we discuss how the complexity results can be improved under this additional sharpness condition. 

The next two theorems present the complexity results of Algorithm \ref{alg.main} with the constant step-size and optimal step-size under the sharpness condition.

\begin{thm}\label{thm:sharp_const}
Suppose Assumptions~\ref{ass:fixed_recourse}--\ref{ass:bundle_func} hold, and $f(x)$ is a $\mu$-sharp function. Consider Algorithm \ref{alg.main} with a constant step-size policy, i.e., $\rho_k=\rho := \frac{\beta\mu^2v}{2\bar\epsilon} > 0$ for all outer iterations $k\geq 0$, where $v\in (0,1)$ is a parameter of the algorithm. Denote 
\begin{equation*}
{\small
\delta^{SC} := \frac{4\bar\epsilon}{\beta} \quad \text{and} \quad K^{SC} := \begin{cases}
\max\left\{1,\left\lceil \frac{f(x_{0,0}) - f^* - \frac{4\beps}{\beta}}{\left(\frac{1}{v} - 1\right)\beps} \right\rceil + 1\right\} = \mathcal{O}\left(\frac{1}{\beps}\right) & \text{if }\frac{1}{2}\leq v < 1, \vspace{+1.5pt}\\
\max\left\{1,\left\lceil \frac{f(x_{0,0}) - f^* - \frac{4\beps}{\beta}}{\left(\frac{1}{v} - 1\right)\beps} \right\rceil + 1\right\} + \left\lceil -\log_{\left(1-\frac{\beta}{2}\right)}\left( \frac{1}{v} - 1 \right) \right\rceil + 1 = \mathcal{O}\left(\frac{1}{\beps}\right) &\text{if }0 < v < \frac{1}{2},
\end{cases}}
\end{equation*}
where $\beps$ is defined in \eqref{eq:beps} and $f^*$ is the optimal function value of $f:X\to\mathbb{R}$, then there exists an outer iteration $k \leq K^{SC}$ such that $f(\xk{0}) - f^* \leq \delta^{SC}$. Furthermore, the total number of inner steps (including null steps) before such a solution is found is at most 
\begin{equation*}
\begin{cases}
\max\left\{1,\left\lceil \frac{f(x_{0,0}) - f^* - \frac{4\beps}{\beta}}{\left(\frac{1}{v} - 1\right)\beps} \right\rceil + 1\right\} \cdot \left(\left\lceil{\frac{16}{(1-\beta)^2}\left(\frac{G^2}{(1-v)\mu^2} - 1\right)}\right\rceil + 1\right) = \mathcal{O}\left(\frac{1}{\beps}\right) & \text{if }\frac{1}{2}\leq v < 1, \\
\max\left\{1,\left\lceil \frac{f(x_{0,0}) - f^* - \frac{2\beps}{\beta v}}{\left(\frac{1}{v} - 1\right)\beps} \right\rceil + 1\right\} \cdot \left(\left\lceil{\frac{16}{(1-\beta)^2}\left(\frac{G^2}{(1-v)\mu^2} - 1\right)}\right\rceil + 1\right)  \\ \quad\quad\quad\quad\quad\quad\quad\quad\quad\quad\quad\quad\quad\quad\quad\quad\quad\quad\quad\quad\quad\quad + \frac{32G^2}{\mu^2v\beta(1-\beta)^2} = \mathcal{O}\left(\frac{1}{\beps}\right) &\text{if }0 < v < \frac{1}{2}.
\end{cases}
\end{equation*}
\end{thm}

\begin{thm}\label{thm:sharp_ideal}
Suppose Assumptions~\ref{ass:fixed_recourse}--\ref{ass:bundle_func} hold, and $f(x)$ is a $\mu$-sharp function. Consider Algorithm \ref{alg.main} taking $\rho_k= \frac{\mu^2}{\hf_k(\xk{0})-\epsilon_2-f^*} > 0$ for all outer iterations $k\geq 0$. Let's define
\begin{equation}\label{eq:ideal_neighborhood_sharp}
\delta^{SI} := \frac{4\bar\epsilon}{\beta} \quad \text{and} \quad K^{SI} := \max\left\{0,\left\lceil-\log_{\left(1-\frac{\beta}{2}\right)}\left(\frac{f(x_{0,0}) - f^* - \frac{3\beps}{\beta}}{\frac{\beps}{\beta}}\right)\right\rceil\right\} = \mathcal{O}\left(\log\left(\frac{1}{\beps}\right)\right)
\end{equation}
where $\beps$ is defined in \eqref{eq:beps} and $f^*$ is the optimal function value of $f:X\to\mathbb{R}$, then there exists an outer iteration $k \leq K^{SI}$ such that $f(\xk{0}) - f^* \leq \delta^{SI}$. Furthermore, the total number of bundle steps (including null steps) before such a solution is found is at most
\begin{equation*}
\left(\frac{16G^2}{(1-\beta)^2\mu^2\left(1 - \frac{2\beta}{3(\beta+1)} \right)}\right) \cdot K^{SI} = \mathcal{O}\left(\log\left(\frac{1}{\beps}\right)\right).
\end{equation*}
\end{thm}

Comparing Theorems~\ref{thm:const_main}--\ref{thm:sharp_ideal}, we can see that the sharpness condition (Definition~\ref{def:sharpness}) improves the convergence property of Algorithm~\ref{alg.main} for both constant and optimal step-size policies. For a constant step-size policy,
Theorem \ref{thm:const_main} shows that Algorithm \ref{alg.main} needs $\mathcal{O}\left(\frac{1}{\beps^3}\right)$ number of inner iterations to reach an $\mathcal{O}\left(\beps\right)$-optimal iterate, while with the sharpness condition, Theorem \ref{thm:sharp_const} shows that Algorithm~\ref{alg.main} needs $\mathcal{O}\left(\frac{1}{\beps}\right)$ total inner iterations to find an $\mathcal{O}(\beps)$-optimal iterate. For the optimal step-size policy, Theorem \ref{thm:ideal_main} shows that Algorithm \ref{alg.main} needs $\mathcal{O}\left(\frac{1}{\beps^2}\right)$ number of inner iterations to reach an $\mathcal{O}\left({\beps}\right)$-optimal iterate, while with the sharpness condition, Theorem \ref{thm:sharp_ideal} shows that the Algorithm~\ref{alg.main} only needs $\mathcal{O}\left(\log(\frac{1}{\beps})\right)$ total inner iterations to a similar iterate.

The proofs of Theorem \ref{thm:sharp_const} and Theorem \ref{thm:sharp_ideal} follow the same proof structure as that of Theorem \ref{thm:const_main} and  are presented in Appendix~\ref{sec.appendix_conv_sharpness}.

\subsection{Sample Complexity for Optimal Step-size Policy}\label{sec.sample_const}

The theoretical results in the previous sections assume the error term $\epsilon_1$ and $\epsilon_2$ are fixed across all function estimations (Assumption \ref{ass:noise_fg}). As such, the quality of the solution found by the algorithm depends on the value of $\epsilon_1$ and $\epsilon_2$. In practice, one can utilize more samples in every outer iteration to obtain a better approximation of the function $f(x)$. This way, one can find a solution with an arbitrarily small optimality gap. In this section, we present the number of samples Algorithm \ref{alg.main} with the optimal step-size needs to access to find a solution with a certain accuracy in high probability. We comment that one can also obtain similar sample complexity results for the other step-size policy studied in Section \ref{sec.const_step_policy} and Section \ref{sec.conv_sharpness}.

For ease of presentation, we define 
\begin{equation}\label{eq:f_xi}
f(x;\xi) = c^Tx + Q(x;\xi),
\end{equation}
where $Q(x;\xi)$ is from \eqref{prob.second_stage}. Then, $f(x;\xi)$ is an unbiased estimator of $f(x)$, i.e., $\mathbb{E}_{\xi}[f(x;\xi)] = f(x)$ for all $x\in X$. In this section, we assume that $f(x;\xi)$ comes from a sub-Gaussian distribution for any $x$, and there is a uniform bound on the variance of the sub-Gaussian distribution over $x\in X$, namely:

\begin{ass}\label{ass:sample_fg}
There is a constant $\sigma > 0$ such that for all $x\in X$ and $\tau > 0$ that
\begin{align}\label{eq:subGau}
\mathbb{E}[|f(x;\xi) - f(x)|^2] \leq \sigma^2\ , \text{and}\ \ 
\mathbb{P}[|f(x;\xi) - f(x)| \geq \tau\sigma] \leq 2{\rm exp}\left(-\frac{\tau^2}{2}\right)\quad.
\end{align}
\end{ass}

We comment that Assumption \ref{ass:sample_fg} is a natural assumption. Assuming a sub-Gaussian distribution essentially means that the samples concentrate around the mean value of the distribution. Suppose $f(x;\xi)$ has bounded value for any $\xi$ and $x$ (which is often the case for practical two-stage stochastic programming), then the sub-Gaussian distribution condition in Assumption~\ref{ass:sample_fg} is automatically satisfied. Since $X$ is a bounded and closed region (Assumption \ref{ass:fixed_recourse}), the uniform upper bound of the variance also naturally holds.

The next theorem presents a high probability convergence result of Algorithm \ref{alg.main} with an optimal step-size policy:
\begin{thm}\label{thm:const_sample}
Suppose Assumptions~\ref{ass:fixed_recourse},~\ref{ass:bound_hg},~\ref{ass:bundle_func} and~\ref{ass:sample_fg} hold. Consider Algorithm \ref{alg.main} with the optimal step-size. For any $\tilde{\epsilon}\in(0,1)$ and $\tau > 0$, we utilize the optimal step-size $\rho_k=\frac{\hf_k(\xk{0})-\tau\tilde{\epsilon}-f^*}{D^2}$ and randomly select a sample set $S_k$ with cardinality $|S_k| = \max\left\{1, \frac{\sigma^2}{\tilde{\epsilon}^2}\right\}$  at the $k$-th outer iteration. Denote
\begin{equation}\label{eq:const_neighborhood_sample}
{\delta}^S := \frac{8(\beta+1)\tilde{\epsilon}\tau}{\beta} \quad \text{and} \quad K^S := \max\left\{0, \left\lceil \frac{\log\pran{\frac{f(x_{0,0})-f^*}{{\delta}^S}}}{-\log\pran{1-\frac{\beta}{4}}} \right\rceil\right\},
\end{equation}
where $f^*$ is the optimal function value of $f:X\to\mathbb{R}$. Then, there exists an outer iteration $k \leq K^S$ such that $f(\xk{0}) - f^* \leq {\delta}^S$ with a probability at least $1-6K^S{\rm exp}\left(-\frac{\tau^2}{2}\right)$.
Furthermore, with a probability at least $1-6K^S{\rm exp}\left(-\frac{\tau^2}{2}\right)$, the total number of inner steps (including null steps) before such a solution is found is at most
\begin{equation}\label{eq:total-null-const}
\frac{8G^2D^2}{3(1-\beta)^2(2-\frac{\beta}{4})\tau^2\tilde\epsilon^2} =\mathcal{O}\pran{\frac{1}{\tau^2\tilde{\epsilon}^2}}\ .
\end{equation}
\end{thm}

We note that Theorem~\ref{thm:const_sample} holds for any $\tilde{\epsilon}\in(0,1)$ and $\tau>0$. From \eqref{eq:const_neighborhood_sample} we know that by varying $\tilde{\epsilon}$ and $\tau$, Algorithm \ref{alg.main} can  identify a solution where the value is arbitrarily close to the optimal solution value. The better the quality of the solution, the smaller the value $\tilde{\epsilon}$ should be, and the more total inner iterations the algorithm may require. As a direct consequence of Theorem \ref{thm:const_sample}, we can obtain a high-probability sample complexity result for Algorithm \ref{alg.main}:

\begin{cor}\label{cor:sample_complexity}
Under the same conditions of Theorem~\ref{thm:const_sample}, for any given parameter $\zeta\in(0,1)$ and optimality tolerance $\delta^S \in (0,1)$, by choosing $\tau = \max\left\{1,\sqrt{2\log\left(\frac{\max\left\{1,6\cdot\left\lceil\log_{\left(1-\frac{\beta}{4}\right)}\left(\frac{\delta^S}{f(x_{0,0}) - f^*}\right)\right\rceil \right\}}{\zeta}\right)}\right\}$ and $\tilde\epsilon = \frac{\beta\delta^S}{8(\beta+1)\tau}$, then with probability at least $1-\zeta$, Algorithm~\ref{alg.main} requires
\begin{align*}
\frac{256G^2D^2(\beta+1)^2}{3(1-\beta)^2(2-\frac{\beta}{4})\beta^2\left(\delta^S\right)^2} = \mathcal{O}\left(\left(\frac{1}{\delta^S}\right)^2\right)
\end{align*}
number of bundle steps (including null steps) and 
\begin{align*}
\max\left\{\frac{256G^2D^2(\beta+1)^2}{3(1-\beta)^2(2-\frac{\beta}{4})\beta^2\left(\delta^S\right)^2}, \frac{16384\sigma^2G^2D^2(\beta+1)^4\tau^2}{3(1-\beta)^2(2-\frac{\beta}{4})\beta^4\left(\delta^S\right)^4}\right\} = \mathcal{O}\left(\frac{\log\left(\log\left(\frac{1}{\delta^S}\right)\right)}{\left(\delta^S\right)^4}\right)
\end{align*}
number of $f(x;\xi)$ samples to find an iterate $x_{k,0}\in X$ satisfying 
\begin{equation*}
f(x_{k,0}) - f^* \leq {\delta}^S. 
\end{equation*}
\end{cor}

\section{A Numerical Study}\label{sec.num_results}

In this section, we present numerical experiments to demonstrate the numerical performance of Algorithm~\ref{alg.main}. We test the performance of different algorithms on seven classic instances of two-stage SLP from various sources (\cite{ShapHome98,LindShapWrig06,OlivSagaSche11}). The \texttt{SH10} and \texttt{SH31} instances are stored in .jl files, while the other five instances are stored in .smps files. The basic statistics of the test problems are listed in Table~\ref{table.prob_info}. Throughout the numerical experiments of Algorithm~\ref{alg.main}, we utilize the bundle function \eqref{eq:limited-memory} with $J_{k,t+1}^g = \{\max\{0,t-3\},\ldots,t+1\}$ and $J_{k,t+1}^s = \{\max\{1,t-3\},\ldots,t+1\}$ for all $k\geq 0$ and $t\geq 0$. In Section~\ref{sec.num_validation}, we test the convergence performance of Algorithm~\ref{alg.main} with constant, optimal, and practical step-size policies discussed in Section~\ref{sec.perf_guarantee}. In Section~\ref{sec.num_comparison}, we compare our proposed algorithm with two alternative solvers, an L-shaped method implementation in \cite{BielJoha22} and the Gurobi solver \citep{Guro23}. 

\begin{table}[h]
\caption{Basic statistics of the test instances used in numerical experiments.}\label{table.prob_info}
\centering
{\footnotesize
\begin{tabular}{|l|l|l|l|l|l|}
\hline
Problem & \begin{tabular}[c]{@{}l@{}} $\#$ of Variables \\ (First Stage) \end{tabular} & \begin{tabular}[c]{@{}l@{}} $\#$ of Constraints \\ (First Stage) \end{tabular} & \begin{tabular}[c]{@{}l@{}} $\#$ of Variables \\ (Second Stage) \end{tabular} & \begin{tabular}[c]{@{}l@{}} $\#$ of Constraints \\ (Second Stage) \end{tabular} & $\#$ of Scenarios \\ \hline
\texttt{20term} & 63 & 3 & 764 & 124 & $1.09951\times 10^{12}$ \\ \hline
\texttt{Gbd} & 17 & 4 & 10 & 5 & 646425 \\ \hline
\texttt{LandS} & 4 & 2 & 12 & 7 & $1.0\times 10^6$ \\ \hline
\texttt{SH10} & 10 & 5 & 15 & 10 & $10000$ \\ \hline
\texttt{SH31} & 10 & 5 & 15 & 10 & $10000$ \\ \hline
\texttt{SSN} & 89 & 1 & 706 & 175 & $1.01750556 \times 10^{70}$ \\ \hline
\texttt{Storm} & 121 & 185 & 1259 & 528 & $6.018531 \times 10^{81}$ \\ \hline
\end{tabular}}
\end{table}

\subsection{Numerical Performance of Algorithm~\ref{alg.main}}\label{sec.num_validation}

In this subsection, we present numerical results of Algorithm~\ref{alg.main} with different step-size and sample sizes on problem \texttt{SH31} to demonstrate how parameters affect the performance of the algorithm.

Figure~\ref{fig.stepsize_comp} shows the optimality gap versus the number of outer iterations of the algorithm with three different step-size choices:
\begin{itemize}
    \item \textbf{Constant step-size:} $\rho_k = \rho$ for all $k\geq 0$, with $\rho\in \{1000,100,10,1,0.1\}$.
    \item \textbf{Optimal step-size:} $\rho_k = C_I\cdot (f(\xk{0}) - f^*)$ for all $k\geq 0$ with $C_I\in \{1000,100,10,1,0.1\}$.
    \item \textbf{Practical step-size:} $\rho_{k}$ is defined by \eqref{eq:practical_stepsize_policy} with $C_P\in \{1000,100,10,1,0.1\}$.
\end{itemize}
In this experiment, the sample size for each outer iteration is fixed at $|S_k| = 100$,  and we choose the descent parameter $\beta = 0.5$ and set the budget of the total inner iterations to $1000$. Here are a few observations: (i) Algorithm \ref{alg.main} has strong numerical performance with all three step-size rules. The best step-size in each case can identify a solution within $10^{-3}$ accuracy within 100 number of outer iterations;
(ii) The best solution obtained by the constant, optimal, and practical step-size has a similar order of error. This is also consistent with our theory that both constant and optimal step-size can find an $\mathcal{O}(\beps)$-optimal solution (see Theorems~\ref{thm:const_main} and~\ref{thm:ideal_main}); 
(iii) Under the constant step-size rule, as long as $\rho$ is not too large, $\mathcal{O}(\beps)$-optimal solutions could always be found within the iteration budget. Since the test problem \texttt{SH31} is a two-stage SLP with a finite support, which is also a sharp function, this observation is consistent with the result of Theorem \ref{thm:sharp_const}; 
(iv) The performance of the practical step-size algorithm is as competitive as the optimal step-size algorithm.

\begin{figure}[ht]
  \centering
  \includegraphics[width=0.3\textwidth]
{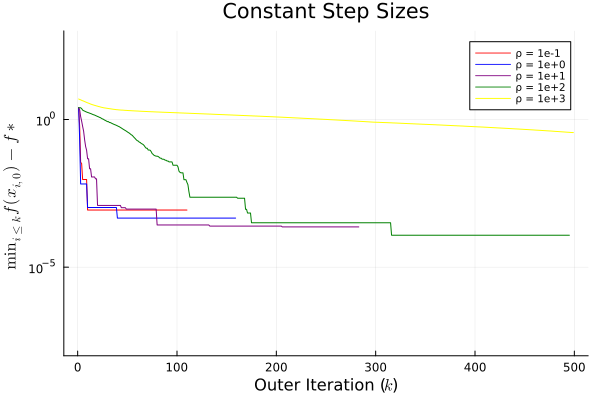}\qquad 
\includegraphics[width=0.3\textwidth]
{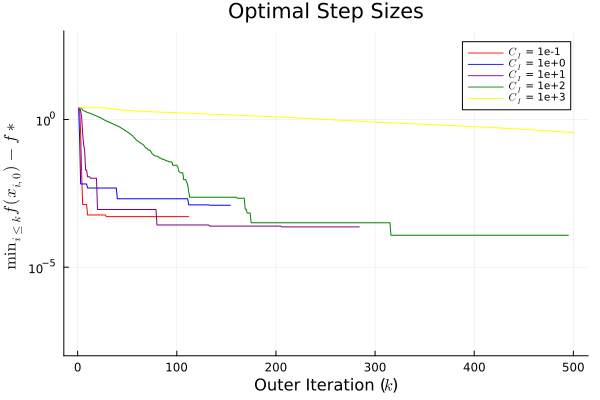}\qquad
\includegraphics[width=0.3\textwidth]
{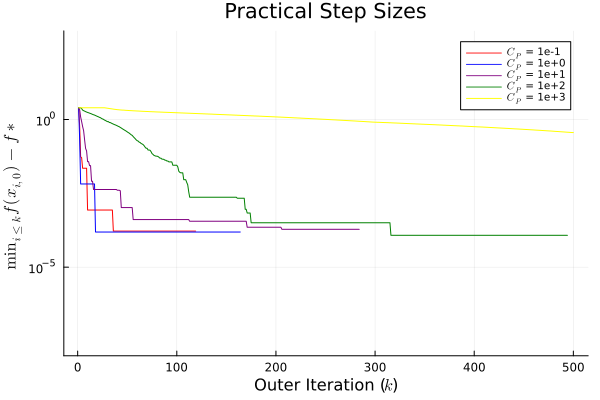}
  \caption{Plots showing the optimality gap versus the number of outer iterations of Algorithm \ref{alg.main} on problem \texttt{SH31} for constant (left), optimal (middle), and practical (right) step-size policies, with sample size being $|S_k| = 100$.}
  \label{fig.stepsize_comp}
\end{figure}

Figure~\ref{fig.scen_comp} reports the optimality gap versus outer iteration of Algorithm \ref{alg.main} with constant step-size policy and different sample sizes for each outer iteration. As we can see, for all of the step-size value, the more samples used in the function evaluation,  the better the solution obtained by the algorithm. This is consistent with Theorem \ref{thm:const_main}: more samples used in the algorithm lead to a smaller $\beps$, and can improve the performance of the algorithm.

\begin{figure}[ht]
  \centering
  \includegraphics[width=0.45\textwidth]
{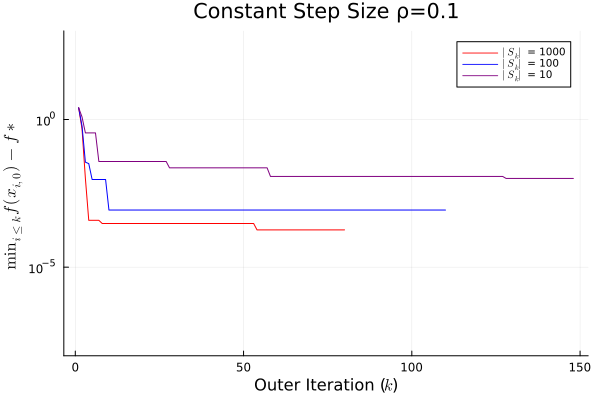}\qquad \includegraphics[width=0.45\textwidth]
{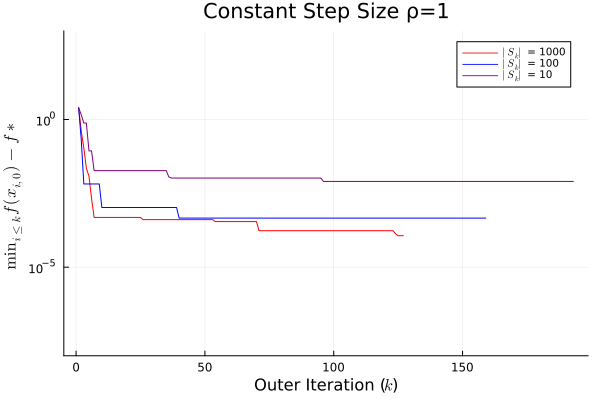} \\
\includegraphics[width=0.45\textwidth]
{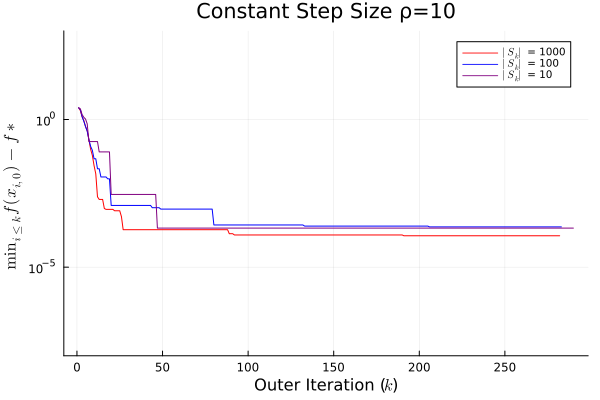}\qquad
\includegraphics[width=0.45\textwidth]
{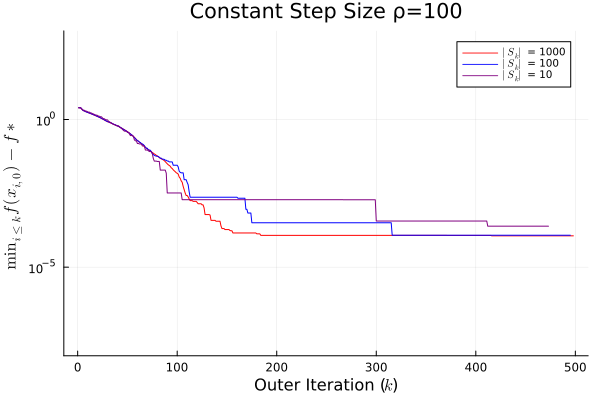}
  \caption{Plots of the optimal objective function gap from the best iterate ever found on problem \texttt{SH31} for different constant step sizes: $\rho = 0.1$ (top left), $\rho = 1$ (top right), $\rho = 10$ (bottom left), and $\rho = 100$ (bottom right).}
  \label{fig.scen_comp}
\end{figure}

\subsection{Comparison with an L-shaped Method and Gurobi}\label{sec.num_comparison}

In this subsection, we compare Algorithm~\ref{alg.main} with an L-shaped method implementation (\cite{BielJoha22}) and the Gurobi solver (\cite{Guro23}) on the test problems listed in Table~\ref{table.prob_info} to demonstrate the efficiency and effectiveness of our proposed method. All methods have been run with a \texttt{CPU} time budget of 12 hours and a memory budget of 10\texttt{GB}. 

For Algorithm~\ref{alg.main}, in addition to the aforementioned budgets on time and memory, we set an extra iteration budget of $10,000$ inner iterations. For these instances, we may not know $f^*$ for large-scale instances before running the algorithm, so here we only report the performance of the algorithm with the constant step-size policy and the practical step-size policy. For the constant step-size policy, we use $\rho\in\{100,10,1,0.1\}$; for the practical step-size policy, we use $C_P\in \{100,10,1,0.1\}$. The sample size $|S_k|$ of the algorithm is set as $100$, namely, for each outer iteration $k$, we randomly sample $100$ scenarios to construct the function estimate and the (sub)gradient estimate. We set the descent parameter $\beta = 0.5$. We run the algorithm with 10 random seeds and report the average performance of the algorithm with the best step-size choice. More specifically, for each step-size choice and each random seed, we run Algorithm \ref{alg.main}, and record the last 50 inner iterates of the algorithm. Then, we sample a new set of $1000$ scenarios (i.e., a larger batch of independent scenarios) to evaluate the objective value of these last 50 iterates, and compute the smallest evaluated objective value over the 50 iterates with the $1000$ scenario as a progress metric of the performance. We report the average and the $95\%$ confidence interval of this progress metric for the algorithm with the best step-size choice for constant step-size policy and practical step-size policy in Table~\ref{table.opt_obj_est}.

We report lower bound estimates of optimal objective function values utilizing a similar approach as in \cite{LindShapWrig06}. More specifically, for each problem listed in Table~\ref{table.prob_info}, we randomly select $100$ scenarios to generate a function estimation. Then, we formulate the function estimation as a linear programming problem. We generate 50 function estimations (by using 5000 scenarios in total), used Gurobi to solve them, and report the average value and $95\%$-confidence intervals of lower bounding estimates in Table~\ref{table.opt_obj_est}. One can show that on average, this approach provides a valid (and perhaps strict) lower bound of the optimal objective value.

We compare our numerical results with \texttt{StochasticPrograms.jl} (\cite{BielJoha22}), a solver for two-stage stochastic programming in Julia, and the Gurobi solver (\cite{Guro23}). \texttt{StochasticPrograms.jl} is an implementation of a certain L-shaped algorithm, and can directly solve a .jl and a .smps file. For Gurobi, we reformulate the two-stage SLP as a big linear programming before solving. These two approaches are perhaps the state-of-the-art solvers for two-stage SLP.

Table~\ref{table.opt_obj_est} reports the numerical results, where ``---$"$ represents that the corresponding approach fails to solve the test problem within the aforementioned time and memory budget. As we can see, \texttt{StochasticPrograms.jl} and Gurobi can only solve the two smallest instances within the budget constraint, while Algorithm \ref{alg.main} with constant and practical step-size policies is capable of being used for solving larger instances. Furthermore, for most of the instances (except \texttt{SSN} and practical step-size policy for \texttt{Storm}), Algorithm \ref{alg.main} can obtain a solution that is reasonably close to the lower bound estimate. Indeed, for $\texttt{SSN}$, it is unclear to us whether the optimality gap comes from the algorithm performance or the lower bound estimation, since the lower bound estimation may lead to a strict lower bound of the optimal value. 
Meanwhile, for \texttt{Storm}, we guess Algorithm~\ref{alg.main} with practical step sizes does not perform as well as the constant step-size policy because the problem has too many scenarios which makes the practical step-size algorithm hard to estimate $f^*$ at a high accuracy.
We also observe that the reported upper bound of the algorithm may be smaller than the lower bound estimates for some instances such as \texttt{Gbd}, \texttt{SH10}, \texttt{SH31}. This is mostly due to the randomness of both the algorithm and the lower bound estimates.

\begin{table}[h]
\caption{Averaged optimal objective function values/estimates}\label{table.opt_obj_est}
\centering
{\tiny
\begin{tabular}{|l|c|c|c|c|c|}
\hline
& \begin{tabular}[c]{@{}c@{}} Algorithm~\ref{alg.main} with \\ Constant Step Sizes \end{tabular} & \begin{tabular}[c]{@{}c@{}} Algorithm~\ref{alg.main} with \\ Practical Step Sizes \end{tabular} & \begin{tabular}[c]{@{}c@{}} Lower Bound \\ Estimates \end{tabular}  & \texttt{StochasticPrograms.jl} & {Gurobi} \\ \hline
\texttt{20term} & 255551.563 $\pm$ 551.646 & 257667.760 $\pm$ 1181.913 &  254415.431 $\pm$ 229.594 & --- & --- \\ \hline
\texttt{Gbd} & 1659.987 $\pm$ 10.944 & 1660.352 $\pm$ 10.883 &  1668.121 $\pm$ 13.775 & --- & --- \\ \hline
\texttt{LandS} & 226.689 $\pm$ 0.808 & 226.690 $\pm$ 0.808 & 226.306 $\pm$ 1.633 & --- & --- \\ \hline
\texttt{SH10} & 15.148 $\pm$ 0.014 & 15.147 $\pm$ 0.014 & 15.153 $\pm$ 0.018  & 15.154 & 15.152 \\ \hline
\texttt{SH31} & 24.980 $\pm$ 0.151 & 24.979 $\pm$ 0.152 & 25.004 $\pm$ 0.216  & 24.988 & 24.988 \\ \hline
\texttt{SSN} & 9.756 $\pm$ 0.278 & 9.755 $\pm$ 0.277 & 7.260 $\pm$ 0.434 & --- & --- \\ \hline
\texttt{Storm} & 15500662.539 $\pm$ 6610.305 & 17736732.210 $\pm$ 748981.205 & 15498615.274 $\pm$ 9690.737 & --- & --- \\ \hline
\end{tabular}}
\end{table}

\bibliographystyle{plainnat}
{\small{\bibliography{references,Lu-papers}}}

\newpage 

\section{Appendix}
\subsection{Supplemental Proofs for Section~\ref{sec.const_step_policy}}\label{app:proof_const}
To begin with, let's prove Lemma~\ref{lem:proximal_bound}, which provides a lower bound for the proximal gap $\Delta_k$.

\begin{proof}[Proof of Lemma~\ref{lem:proximal_bound}]
It follows from \eqref{eq.Delta_k} that for any $x\in X$,
\begin{equation}\label{eq.Delta_lb}
\Delta_k\ge f(\xk{0})-\pran{f(x)+\frac{\rho_k}{2}\|x-\xk{0}\|^2}.
\end{equation}
Given any $k\geq 0$, if $f(\xk{0})-f^* > \rho_k D^2$, we select $x := x^*\in X$ in \eqref{eq.Delta_lb} such that $f(x^*) = f^*$, then it follows Assumption~\ref{ass:fixed_recourse} that
\begin{equation*}
\begin{aligned}
\Delta_k &\geq f(\xk{0}) - \left(f^* + \frac{\rho_k}{2}\|x^*-\xk{0}\|^2\right) \\
&\geq \frac{f(\xk{0}) - f^*}{2} + \frac{f(\xk{0}) - f^* - \rho_kD^2}{2} > \frac{f(\xk{0}) - f^*}{2}.
\end{aligned}
\end{equation*}
Otherwise, when $f(\xk{0})-f^* \leq \rho_k D^2$, let's define $\lambda_k = \frac{f(\xk{0})-f^*}{\rho_k D^2} \in [0,1]$ and $x^*\in X$ such that $f(x^*) = f^*$. By choosing $x = \lambda_k x^* + (1-\lambda_k)\xk{0} \in X$ (since $X$ is a convex set) in \eqref{eq.Delta_lb}, from the convexity of $f$, we know $f(x) \leq \lambda_kf^* + (1-\lambda_k)f(\xk{0})$. Moreover, it follows from Assumption~\ref{ass:fixed_recourse} and \eqref{eq.Delta_lb} that
\begin{equation*}
\begin{aligned}
\Delta_k &\geq f(\xk{0})-\pran{f(x)+\frac{\rho_k}{2}\|x-\xk{0}\|^2} \\
&\geq \lambda_k(f(\xk{0})-f^*)-\frac{\rho_k \lambda_k^2}{2}\|\xk{0}-x^*\|^2 \\
&\geq \lambda_k(f(\xk{0})-f^*)-\frac{\rho_k \lambda_k^2}{2}D^2 
=\frac{(f(\xk{0})-f^*)^2}{2\rho_kD^2},
\end{aligned}
\end{equation*}
which finishes the proof.
\end{proof}

Next, we provide the proof of Lemma~\ref{lem:Deltakt_connection}.

\begin{proof}[Proof of Lemma~\ref{lem:Deltakt_connection}]

To start with, we are going to show that $\|s_{k,t}\| := \|\rho_k(\xk{0} - \xk{t})\| \leq G$ for any $k\geq 0$ and $1 \leq t \leq T_k$. By the update of $\xk{t}$ in Algorithm~\ref{alg.main}, the convexity of bundle function $f_{k,t-1}$ (see Assumption~\ref{ass:bundle_func} and the definition of $f_{k,0}$ in Algorithm~\ref{alg.main}), and the fact of $\{\xk{0},\xk{t}\}\subset X$, we know that for any $k\geq 0$ and $1 \leq t\leq T_k$
\begin{align*}
f_{k,t-1}(\xk{t}) + \frac{\rho_k}{2}\|\xk{t} - \xk{0}\|^2 + \frac{\rho_k}{2}\|\xk{t} - \xk{0}\|^2 \leq f_{k,t-1}(\xk{0}),
\end{align*}
which further implies that
\begin{align*}
G\|\xk{t} - \xk{0}\| \geq f_{k,t-1}(\xk{0}) - f_{k,t-1}(\xk{t}) \geq \rho_k\|\xk{t} - \xk{0}\|^2 = \|s_{k,t}\|\|\xk{t} - \xk{0}\|,
\end{align*}
where the first inequality follows Assumptions~\ref{ass:bound_hg} and~\ref{ass:bundle_func} and the last equality applies the definition of $s_{k,t}$. Therefore, we complete the first part of the proof.

Next, we are going to prove \eqref{eq:combined_both}. For any $k \geq 0$ and $1\leq t \leq T_k-1$, let's define a function $\tfk{t}:\RR^{n}\to\RR$ as
\begin{equation}\label{eq:def_tildefk}
\tfk{t}(x)=\max\left\{\fk{t-1}(\xk{t})+\sk{t}^T(x-\xk{t}), \hf_{k}(\xk{t})+\hg_{k}(\xk{t})^T(x-\xk{t})\right\}.
\end{equation}

Then it follows from Assumption~\ref{ass:bundle_func} that $\tfk{t}(x)\le \fk{t}(x)$ for all $x\in\RR^n$. Let's denote
\begin{equation*}
\yk{t+1}=\arg\min_{x\in\RR^n}\left\{\tfk{t}{(x)} +\frac{\rho_{k}}{2}\|x-\xk{0}\|^2\right\},
\end{equation*}
which is unique since $\rho_k>0$ and $\tfk{t}$ is convex.  Using $\theta_{k,t}$ in the lemma statement, let's define 
\begin{equation}\label{eq:theta-y}
\zk{t+1}=\xk{0}-\frac{1}{\rho_k}\pran{\theta_{k,t} \hg_{k}(\xk{t}) + (1-\theta_{k,t})\sk{t}}.
\end{equation} 
Next, from the uniqueness of $\yk{t+1}$, we are going to prove $\yk{t+1} = \zk{t+1}$ by showing $\rho_k(\xk{0} - \zk{t+1}) \in \partial \tfk{t}(\zk{t+1})$. We consider two cases as follows.
\begin{itemize}
\item[(a)] If $\rho_k(\hf_k(\xk{t})-\fk{t-1}(\xk{t})) \geq \|\hg_{k}(\xk{t})-\sk{t}\|^2$, then $\theta_{k,t} = 1$ and $\zk{t+1} = \xk{0} - \frac{1}{\rho_k}\hg_k(\xk{t})$ by \eqref{eq:theta-y}. We further have
\begin{equation}\label{eq:compare_case_a}
\begin{aligned}
&\hf_k(\xk{t}) - \fk{t-1}(\xk{t}) + (\hg_k(\xk{t}) - s_{k,t})^T(\zk{t+1} - \xk{t}) \\
\geq \ &\frac{1}{\rho_k}\|\hg_{k}(\xk{t})-\sk{t}\|^2 + (\hg_k(\xk{t}) - s_{k,t})^T\left(\xk{0} - \xk{t} - \frac{1}{\rho_k}\hg_k(\xk{t})\right) \\
= \ &(\hg_k(\xk{t}) - s_{k,t})^T\left( \xk{0} - \xk{t} - \frac{1}{\rho_k}s_{k,t} \right) = 0,
\end{aligned}
\end{equation}
where the last equality comes from the definition of $s_{k,t}$. In addition, \eqref{eq:compare_case_a} implies that
\begin{equation*}
\tfk{t}(\zk{t+1}) = \hf_{k}(\xk{t})+\hg_{k}(\xk{t})^T(\zk{t+1}-\xk{t});
\end{equation*}
see \eqref{eq:def_tildefk}. Then we further know that $\rho_k(\xk{0} - \zk{t+1}) = \hg_{k}(\xk{t}) \in \partial \tfk{t}(\zk{t+1})$.
\item[(b)] Otherwise, when $\rho_k(\hf_k(\xk{t})-\fk{t-1}(\xk{t})) < \|\hg_{k}(\xk{t})-\sk{t}\|^2$, from \eqref{eq:initialization}, the convexity of function $\hf_k$, and Assumptions~\ref{ass:noise_fg} and~\ref{ass:bundle_func}, we know that $\fk{t}(x) \le \hf_k(x) \leq f(x) + \epsilon_2$ for all $k\geq 0$ and $0\leq t\leq T_k-1$. Therefore, $\theta_{k,t} = \frac{\rho_k(\hf_k(\xk{t})-\fk{t-1}(\xk{t}))}{\|\hg_{k}(\xk{t})-\sk{t}\|^2} \in [0,1)$ by the definition of $\theta_{k,t}$. Moreover,
\begin{equation}\label{eq:compare_case_b}
\begin{aligned}
&\hf_k(\xk{t}) - \fk{t-1}(\xk{t}) + (\hg_k(\xk{t}) - s_{k,t})^T(\zk{t+1} - \xk{t}) \\
= \ &\hf_k(\xk{t}) - \fk{t-1}(\xk{t}) + (\hg_k(\xk{t}) - s_{k,t})^T\left(\xk{0} - \xk{t} - \frac{1}{\rho_k}\pran{\theta_{k,t} \hg_{k}(\xk{t}) + (1-\theta_{k,t})\sk{t}}\right) \\
= \ &\hf_k(\xk{t}) - \fk{t-1}(\xk{t}) + (\hg_k(\xk{t}) - s_{k,t})^T\left(\frac{\theta_{k,t}}{\rho_k}(\sk{t} - \hg_k(\xk{t}))\right) \\
= \ &\hf_k(\xk{t}) - \fk{t-1}(\xk{t}) - \frac{\theta_{k,t}}{\rho_k}\|\hg_k(\xk{t}) - s_{k,t}\|^2 = 0,
\end{aligned}
\end{equation}
where the first two equations use definitions of $\zk{t+1}$ and $\sk{t}$, while the last equation is from the definition of $\theta_{k,t}$. Furthermore, \eqref{eq:compare_case_b} implies that $\partial \tfk{t}(\zk{t+1}) = \textbf{co}(\{\sk{t},\hg_k(\xk{t})\})$, and then we conclude 
\begin{equation*}
\rho_k(\xk{0} - \zk{t+1}) = \theta_{k,t}\hg_k(\xk{t}) + (1-\theta_{k,t})\sk{t} \in \textbf{co}(\{\sk{t},\hg_k(\xk{t})\}) = \partial \tfk{t}(\zk{t+1}).
\end{equation*}
\end{itemize}

Thus, we conclude $\yk{t+1} = \zk{t+1} = \xk{0}-\frac{1}{\rho_k}\pran{\theta_{k,t} \hg_{k}(\xk{t}) + (1-\theta_{k,t})\sk{t}}$ from \eqref{eq:theta-y}. By the facts of $\theta_{k,t}\in[0,1]$ and $\tfk{t}(x)\le \fk{t}(x)$ for all $x\in\RR^n$ and the definitions of $\tfk{t}$ and $\yk{t+1}$, we have 
\begin{equation}\label{eq:big-1}
    \begin{aligned}
        &\fk{t}(\xk{t+1})+\frac{\rho_k}{2}\|\xk{t+1}-\xk{0}\|^2 \\
        \ge \ &\tfk{t}(\xk{t+1})+\frac{\rho_k}{2}\|\xk{t+1}-\xk{0}\|^2 \\
        \ge \ &\tfk{t}(\yk{t+1})+\frac{\rho_k}{2}\|\yk{t+1}-\xk{0}\|^2 \\
        \ge \ &\theta_{k,t}\pran{\hf_k(\xk{t})+\hg_{k}(\xk{t})^T(\yk{t+1}-\xk{t})} + (1-\theta_{k,t})\pran{\fk{t-1}(\xk{t})+\sk{t}^T(\yk{t+1}-\xk{t})}+\frac{\rho_k}{2}\|\yk{t+1}-\xk{0}\|^2 \\
        = \ & \fk{t-1}(\xk{t}) + \theta_{k,t}(\hf_k(\xk{t})-\fk{t-1}(\xk{t}))+(\theta_{k,t}\hg_{k}(\xk{t})+(1-\theta_{k,t})\sk{t})^T(\yk{t+1}-\xk{t}) +\frac{\rho_k}{2}\|\yk{t+1}-\xk{0}\|^2 \\
        = \ & \fk{t-1}(\xk{t}) + \theta_{k,t}(\hf_k(\xk{t})-\fk{t-1}(\xk{t}))+\rho_k(\xk{0}-\yk{t+1})^T(\yk{t+1}-\xk{t}) +\frac{\rho_k}{2}\|\yk{t+1}-\xk{0}\|^2 \\
        = \ & \fk{t-1}(\xk{t}) + \theta_{k,t}(\hf_k(\xk{t})-\fk{t-1}(\xk{t}))+\frac{\rho_k}{2}(\xk{0}-\yk{t+1})^T(\yk{t+1}+\xk{0}-2\xk{t}) \\
        = \ & \fk{t-1}(\xk{t}) + \theta_{k,t}(\hf_k(\xk{t})-\fk{t-1}(\xk{t}))+\frac{\rho_k}{2}\|\xk{t}-\xk{0}\|^2-\frac{\rho_k}{2}\|\yk{t+1}-\xk{t}\|^2.
    \end{aligned}
\end{equation}
Meanwhile, we also have
\begin{equation}\label{eq:norm}
    \begin{aligned}
        \frac{\rho_k}{2}\|\yk{t+1}-\xk{t}\|^2 &= \frac{\rho_k}{2}\left\|\pran{\xk{0}-\frac{1}{\rho_k}\pran{\theta_{k,t} \hg_{k}(\xk{t}) + (1-\theta_{k,t})\sk{t}}} - \pran{\xk{0}-\frac{1}{\rho_k} \sk{t}}\right\|^2 \\
        &=\frac{\theta_{k,t}^2}{2\rho_k}\|\sk{t}-\hg_{k}(\xk{t})\|^2.
    \end{aligned}
\end{equation}

Combining \eqref{eq:big-1} and \eqref{eq:norm}, and recalling the definition of $\tilde\Delta_{k,t}$ in \eqref{eq.inner_prox_gap}, we obtain for any $k\geq 0$ and $1 \leq t\le T_k-1$ that
\begin{equation}
\begin{aligned}
    \tDeltak{t} &= \hf_k(\xk{0}) - \left(\fk{t}(\xk{t+1}) + \frac{\rho_k}{2}\|\xk{t+1} - \xk{0}\|^2\right) \\
    &\le \hf_k(\xk{0}) - \left(\fk{t-1}(\xk{t}) + \theta_{k,t}(\hf_k(\xk{t})-\fk{t-1}(\xk{t}))+\frac{\rho_k}{2}\|\xk{t}-\xk{0}\|^2-\frac{\rho_k}{2}\|\yk{t+1}-\xk{t}\|^2\right) \\
    &= \tDeltak{t-1} -\theta_{k,t}(\hf_k(\xk{t})-\fk{t-1}(\xk{t}))+\frac{\theta_{k,t}^2}{2\rho_k}\|\sk{t}-\hg_{k}(\xk{t})\|^2,
\end{aligned}
\end{equation}
which concludes the statement.
\end{proof}

\subsection{Supplemental Proofs for Section~\ref{sec.conv_sharpness}}\label{sec.appendix_conv_sharpness}

To prove Theorems~\ref{thm:sharp_const} and~\ref{thm:sharp_ideal}, we first present a lemma similar to Lemma~\ref{lem:proximal_bound}, which provides a lower bound for the proximal gap $\Delta_k$ under the additional sharpness condition.

\begin{lem}\label{lem:proximal_bound_sharp}
Suppose $f(x)$ is a $\mu$-sharp function. For any outer iteration $k\geq 0$, from Algorithm~\ref{alg.main} we have
\begin{equation}\label{eq:proximal_bound_sharp}
\Delta_k \ge \left\{ \begin{array}{cl}
    \mu^2/(2\rho_k) & \text{ if } f(\xk{0})-f^* \ge \mu^2/\rho_k, \\
    (f(\xk{0})-f^*)/2 & \text{ otherwise, } 
\end{array} \right.
\end{equation}
where $\rho_k>0$ is the step size at the $k$-th outer iteration and $f^*$ represents the optimal value of function $f:X\to\RR$.
\end{lem}
\begin{proof}
It follows from \eqref{eq.Delta_k} that for any $x\in X$,
\begin{equation}\label{eq.Delta_lb_new}
\Delta_k\ge f(\xk{0})-\pran{f(x)+\frac{\rho_k}{2}\|x-\xk{0}\|^2}.
\end{equation}
If $f(\xk{0})-f^* < \mu^2/\rho_k$, by selecting $x = x^* := \arg\min_{x\in X^*}\|\xk{0} - x\|$ in \eqref{eq.Delta_lb_new}, it follows Definition~\ref{def:sharpness} that
\begin{equation*}
\begin{aligned}
\Delta_k &\geq f(\xk{0}) - \left(f^* + \frac{\rho_k}{2}\|x^*-\xk{0}\|^2\right) \\
&= \frac{f(\xk{0}) - f^*}{2} + \frac{f(\xk{0}) - f^* - \rho_k\|x^* - \xk{0}\|^2}{2} \\
&\geq \frac{f(\xk{0}) - f^*}{2} + \frac{f(\xk{0}) - f^* - (\rho_k/\mu^2)(f(\xk{0}) - f^*)^2}{2} \\
&= \frac{f(\xk{0}) - f^*}{2} + \frac{(f(\xk{0}) - f^*)(1 - (\rho_k/\mu^2)(f(\xk{0}) - f^*))}{2} > \frac{f(\xk{0}) - f^*}{2}.
\end{aligned}
\end{equation*}
Otherwise, when $f(\xk{0})-f^* \geq \mu^2/\rho_k$, let's define $\lambda_k = \frac{\mu^2}{\rho_k(f(\xk{0})-f^*)} \in (0,1]$ and $x^* = \arg\min_{x\in X^*}\|\xk{0} - x\|$. Choosing $x = \lambda_k x^* + (1-\lambda_k)\xk{0}$ it follows from \eqref{eq.Delta_lb_new} that
\begin{equation*}
\begin{aligned}
\Delta_k &\geq f(\xk{0})-\pran{f(x)+\frac{\rho_k}{2}\|x-\xk{0}\|^2} \geq \lambda_k(f(\xk{0})-f^*)-\frac{\rho_k \lambda_k^2}{2}\|\xk{0}-x^*\|^2 \\
&= \frac{\mu^2}{\rho_k} - \frac{\mu^4\|\xk{0} - x^*\|^2}{2\rho_k(f(\xk{0})-f(x^*))^2} \geq \frac{\mu^2}{\rho_k} - \frac{\mu^2}{2\rho_k} = \frac{\mu^2}{2\rho_k},
\end{aligned}
\end{equation*}
where the second inequality follows the definition of $x$ and $f(x) \leq \lambda_kf^* + (1-\lambda_k)f(\xk{0})$ due to the convexity of $f(x)$, the third inequality uses the sharpness of $f(x)$. Therefore, we conclude the statement.
\end{proof}

Now we are ready to prove Theorems~\ref{thm:sharp_const} and~\ref{thm:sharp_ideal}.

\begin{proof}[Proof of Theorem~\ref{thm:sharp_const}]
This proof includes three parts as follows.

\textbf{Part 1 (Length of the inner loop)} \quad With the same argument as the first part in the proof of Theorem~\ref{thm:const_main}, we know that \eqref{eq.Tk_bound} holds.

\textbf{Part 2 (Objective decrease on serious iterates)} \quad It follows the same proof as the second part in the proof of Theorem~\ref{thm:const_main} that \eqref{eq.obj_decay} holds.

\textbf{Part 3 (Total outer and inner iterations)} \quad Suppose $f(\xk{0}) - f^* > \delta^{SC}$. Let's first consider the case when $\frac{1}{2}\leq v < 1$. In this case, it follows from the definition of $\rho$ that $\delta^{SC} = \frac{4\beps}{\beta} \geq \frac{\mu^2}{\rho}$. From Lemma~\ref{lem:proximal_bound_sharp} and \eqref{eq.obj_decay}, we have
\begin{align*}
f(x_{k+1,0}) - f^* &\leq f(\xk{0}) - f^* - \left(\beta \Delta_k - \beps \right) \\
&\leq f(\xk{0}) - f^* - \left(\frac{\beta\mu^2}{2\rho} - \bar\epsilon\right) = f(\xk{0}) - f^* - \left(\frac{1}{v} - 1\right)\beps,
\end{align*}
where the second inequality comes from \eqref{eq:proximal_bound_sharp} and $f(\xk{0})-f^* \ge \mu^2/\rho$. 
Therefore, there are at most 
\begin{equation}\label{eq.sharp_const_iter_bound_1}
\begin{aligned}
\max\left\{0,\left\lceil \frac{f(x_{0,0}) - f^* - \delta^{SC}}{\left(\frac{1}{v} - 1\right)\beps} \right\rceil\right\} = \max\left\{0,\left\lceil \frac{f(x_{0,0}) - f^* - \frac{4\beps}{\beta}}{\left(\frac{1}{v} - 1\right)\beps} \right\rceil\right\}
\end{aligned}
\end{equation}
outer iterations such that $f(\xk{0}) - f^* > \delta^{SC} \geq \frac{\mu^2}{\rho}$. By the result of \textbf{Part 1} (see \eqref{eq.Tk_bound}), the definition of $\beps$, Lemma~\ref{lem:proximal_bound_sharp}, and the selection of constant step size $\rho := \frac{\beta\mu^2 v}{2\beps}$, each such an outer iteration uses at most
\begin{equation}\label{eq.sharp_const_inner_bound_1}
\begin{aligned}
&\left\lceil{\frac{8G^2}{\rho(1-\beta)^2(\Delta_k - \epsilon_1 - \epsilon_2)} - \frac{16}{(1-\beta)^2}}\right\rceil + 1 \leq \left\lceil{\frac{8G^2}{\rho(1-\beta)^2\left(\frac{\mu^2}{2\rho} - \frac{\bar\epsilon}{\beta}\right)} - \frac{16}{(1-\beta)^2}}\right\rceil + 1 \\
= \ &\left\lceil{\frac{8G^2}{\rho(1-\beta)^2\left(\frac{1}{v} - 1\right)\frac{\bar\epsilon}{\beta}} - \frac{16}{(1-\beta)^2}}\right\rceil + 1 = \left\lceil{\frac{16}{(1-\beta)^2}\left(\frac{G^2}{(1-v)\mu^2} - 1\right)}\right\rceil + 1
\end{aligned}
\end{equation}
inner iterations to take a serious step. Combining \eqref{eq.sharp_const_iter_bound_1} and \eqref{eq.sharp_const_inner_bound_1}, we know the total number of inner steps (including null steps) before an iterate $\xk{0}$ satisfying $f(\xk{0}) - f^* \leq \delta^{SC}$ is at most
\begin{equation}\label{eq:sharp_const_complexity_p1}
\max\left\{1,\left\lceil \frac{f(x_{0,0}) - f^* - \frac{4\beps}{\beta}}{\left(\frac{1}{v} - 1\right)\beps} \right\rceil + 1\right\} \cdot \left(\left\lceil{\frac{16}{(1-\beta)^2}\left(\frac{G^2}{(1-v)\mu^2} - 1\right)}\right\rceil + 1\right).
\end{equation}
This proves the theorem in the case that $\frac{1}{2}\le v<1$.

When $0 < v < \frac{1}{2}$, let's denote $T_{SC}$ as the first outer iteration such that $f(\xk{0}) - f^* < \frac{\mu^2}{\rho}$. From the same logic of \eqref{eq.sharp_const_iter_bound_1} and the definition of $\rho := \frac{\beta\mu^2 v}{2\beps}$, we know 
\begin{equation*}
T_{SC} \leq \max\left\{0,\left\lceil \frac{f(x_{0,0}) - f^* - \frac{\mu^2}{\rho}}{\left(\frac{1}{v} - 1\right)\beps} \right\rceil\right\} + 1 = \max\left\{0,\left\lceil \frac{f(x_{0,0}) - f^* - \frac{2\beps}{\beta v}}{\left(\frac{1}{v} - 1\right)\beps} \right\rceil\right\} + 1.
\end{equation*}
Moreover, from \eqref{eq.sharp_const_inner_bound_1} and \eqref{eq:sharp_const_complexity_p1}, each outer iteration $0\leq k < T_{SC}$ needs at most $\left\lceil{\frac{16}{(1-\beta)^2}\left(\frac{G^2}{(1-v)\mu^2} - 1\right)}\right\rceil + 1$ inner iterations to take a serious step, and the total number of inner steps (including null steps) for $0\leq k < T_{SC}$ is at most $\max\left\{1,\left\lceil \frac{f(x_{0,0}) - f^* - \frac{4\beps}{\beta}}{\left(\frac{1}{v} - 1\right)\beps} \right\rceil + 1\right\} \cdot \left(\left\lceil{\frac{16}{(1-\beta)^2}\left(\frac{G^2}{(1-v)\mu^2} - 1\right)}\right\rceil + 1\right)$.

It follows from the definition of $T_{SC}$ that $f(x_{T_{SC},0}) - f^* < \frac{\mu^2}{\rho}$. Next, we show that $f(\xk{0}) - f^* < \frac{\mu^2}{\rho}$ for all $k \geq T_{SC}$. For any iteration $k \geq T_{SC}$ satisfying $f(\xk{0}) - f^* < \frac{\mu^2}{\rho}$,
it follows from Lemma~\ref{lem:proximal_bound_sharp} and the result of \textbf{Part 2} (see \eqref{eq.obj_decay}) that
\begin{equation}\label{eq:thm3_keylem}
f(x_{k+1,0}) - f^* \leq f(\xk{0}) - f^* - (\beta\Delta_k - \beps) \leq f(\xk{0}) - f^* - \left(\frac{\beta}{2}(f(\xk{0}) - f^*) - \bar\epsilon\right).
\end{equation}
Combining \eqref{eq:thm3_keylem} and $f(\xk{0}) - f^* > \delta^{SC} = \frac{4\bar\epsilon}{\beta}$, we know that $f(x_{k+1,0}) < f(\xk{0})$, thus $f(\xk{0})$ monotonically decreases in $k$. This showcases by induction that $f(\xk{0}) - f^* < \frac{\mu^2}{\rho}$ for all $k \geq T_{SC}$ until an iterate $\xk{0}$ satisfying $f(\xk{0}) - f^* \leq \delta^{SC}$ is found. 

Moreover, \eqref{eq:thm3_keylem} implies that
\begin{equation*}
\begin{aligned}
&f(x_{k+1,0}) - f^* - \frac{2\bar\epsilon}{\beta} \leq \left(1 - \frac{\beta}{2}\right)\left(f(\xk{0}) - f^* - \frac{2\bar\epsilon}{\beta}\right).
\end{aligned}
\end{equation*} 
Then it at most takes us a total of $\left\lceil -\log_{\left(1-\frac{\beta}{2}\right)}\left(\frac{\frac{\mu^2}{\rho} - \frac{2\bar\epsilon}{\beta}}{\frac{2\bar\epsilon}{\beta}}\right) \right\rceil = \left\lceil -\log_{\left(1-\frac{\beta}{2}\right)}\left( \frac{1}{v} - 1 \right) \right\rceil$ number of outer iterations after $T_{SC}$ to find an iterate such that $f(\xk{0}) - f^* \leq \delta^{SC}$. We denote $K_{SC}\geq 0$ as the last outer iteration such that $\delta^{SC} < f(x_{K_{SC},0}) - f^* < \frac{\mu^2}{\rho}$. Then for any $k \geq 0$ with $\delta^{SC} < f(\xk{0}) - f^* < \frac{\mu^2}{\rho}$, following Lemma~\ref{lem:proximal_bound_sharp} and the result of \textbf{Part 1} (see \eqref{eq.Tk_bound}), it at most uses
\begin{equation*}
\begin{aligned}
&\left\lceil{\frac{8G^2}{\rho(1-\beta)^2(\Delta_k - \epsilon_1 - \epsilon_2)} - \frac{16}{(1-\beta)^2}}\right\rceil \leq \left\lceil{\frac{8G^2}{\rho(1-\beta)^2\left(\frac{f(\xk{0}) - f^*}{2} - \frac{\bar\epsilon}{\beta}\right)}} - \frac{16}{(1-\beta)^2}\right\rceil \\
\leq \ &{\frac{16G^2}{\rho(1-\beta)^2\left(f(\xk{0}) - f^* - \frac{2\bar\epsilon}{\beta}\right)} } \leq {\frac{16G^2}{\rho(1-\beta)^2\left(1 - \frac{\beta}{2}\right)^{k - K_{SC}}\left(f(x_{K_{SC},0}) - f^* - \frac{2\bar\epsilon}{\beta}\right)} } \\
< \ &\left(1 - \frac{\beta}{2}\right)^{K_{SC} - k}\cdot\left(\frac{16G^2}{\rho(1-\beta)^2\left(\delta^{SC} - \frac{2\bar\epsilon}{\beta}\right)} \right)
\end{aligned}
\end{equation*}
number of inner iterations to take a serious step. Therefore, in this case, the total number of bundle steps (including null steps) is
\begin{equation*}
\begin{aligned}
&\sum_{k = T_{SC}}^{K_{SC}}\left(1 - \frac{\beta}{2}\right)^{K_{SC} - k}\cdot\left(\frac{16G^2}{\rho(1-\beta)^2\left(\delta^{SC} - \frac{2\bar\epsilon}{\beta}\right)} \right) \\
= \ &\left(\frac{16G^2}{\frac{2\beps\rho}{\beta}(1-\beta)^2}\right)  \cdot \left(\frac{1 - \left(1-\frac{\beta}{2}\right)^{K_{SC} - T_{SC} + 1}}{\frac{\beta}{2}}\right) \leq \frac{32G^2}{\mu^2v\beta(1-\beta)^2}.
\end{aligned}
\end{equation*}
Combining the cases above, we conclude the statement.
\end{proof}

\begin{proof}[Proof of Theorem~\ref{thm:sharp_ideal}]
The proof follows the same structure as the proof of Theorem \ref{thm:const_main}. Without loss of generality, we assume $f(x_{0,0})-f^*>\delta^I$, thus $K^I > 0$.

\textbf{Part 1 (Length of the inner loop)} \quad With the same argument as the first part in the proof of Theorem~\ref{thm:const_main}, we know that \eqref{eq.Tk_bound} holds.

\textbf{Part 2 (Objective decrease on serious iterates)} \quad It follows the same proof as the second part in the proof of Theorem~\ref{thm:const_main} that \eqref{eq.obj_decay} holds.

\textbf{Part 3 (Total outer and inner iterations)} \quad By Assumption~\ref{ass:noise_fg} and Lemma~\ref{lem:proximal_bound_sharp}, the optimal step-size policy implies that
\begin{equation}\label{eq:ideal_sharp_delta_lb}
f(\xk{0}) - f^* \geq \hf_k(\xk{0}) - \epsilon_2 - f^* = \frac{\mu^2}{\rho_k}\quad \text{and}\quad \Delta_k \geq \frac{\mu^2}{2\rho_k}.
\end{equation}
From \eqref{eq.obj_decay} and \eqref{eq:ideal_sharp_delta_lb}, we have that
\begin{equation*}
\begin{aligned}
f(x_{k+1,0}) - f^* &\leq f(\xk{0}) - f^* - (\beta \Delta_k - \beps) \leq f(\xk{0}) - f^* - \left(\frac{\beta\mu^2}{2\rho_k} - \beps\right) \\
&= f(\xk{0}) - f^* - \left(\frac{\beta}{2}(\hf_k(\xk{0})-\epsilon_2-f^*) - \beps \right) \\
&\leq f(\xk{0}) - f^* - \left(\frac{\beta}{2}(f(\xk{0}) - \epsilon_1 -\epsilon_2-f^*) - \beps \right) \\
&= \left(1-\frac{\beta}{2}\right)(f(\xk{0}) - f^*) + \frac{\beta}{2}(\epsilon_1 + \epsilon_2) + \beps \\
&< \left(1-\frac{\beta}{2}\right)(f(\xk{0}) - f^*) + \frac{3\beps}{2},
\end{aligned}
\end{equation*}
where the first equality is from the definition of $\rho_k$, the third inequality uses Assumption~\ref{ass:noise_fg}, and the last inequality follows the definition of $\beps$. Therefore, we have
\begin{equation}\label{eq:sharp_ideal_decay}
f(x_{k+1,0}) - f^* - \frac{3\beps}{\beta} < \left(1-\frac{\beta}{2}\right)\left(f(\xk{0}) - f^* - \frac{3\beps}{\beta}\right).
\end{equation}
Let's denote $K^{SI}\geq 0$ as the first iteration such that $f(\xk{0}) - f^* \leq \delta^{SI}$, then it follows that
\begin{align*}
K^{SI} &\leq \max\left\{0,\left\lceil-\log_{\left(1-\frac{\beta}{2}\right)}\left(\frac{f(x_{0,0}) - f^* - \frac{3\beps}{\beta}}{\delta^{SI} - \frac{3\beps}{\beta}}\right)\right\rceil\right\} \\
&= \max\left\{0,\left\lceil-\log_{\left(1-\frac{\beta}{2}\right)}\left(\frac{f(x_{0,0}) - f^* - \frac{3\beps}{\beta}}{\frac{\beps}{\beta}}\right)\right\rceil\right\} = \mathcal{O}\left(\log\left(\frac{1}{\beps}\right)\right).
\end{align*}
From Assumption~\ref{ass:noise_fg}, the definition of $\beps$, and $f(\xk{0}) - f^* > \delta^{SI} = \frac{4\beps}{\beta}$ for any $k < K^{SI}$, we know
\begin{equation}\label{eq.sharp_ideal_func_gap_lb}
\hf_k(\xk{0}) - \epsilon_2 - f^* \geq f(\xk{0}) - f^* - \epsilon_1 - \epsilon_2 > \delta^{SI} - \frac{\beps}{1+\beta} > \frac{3\beps}{\beta}.
\end{equation}
By \eqref{eq:ideal_sharp_delta_lb} and the definition of $\beps$, the number of inner iterations at outer iteration $k$ could be upper bounded by
\begin{equation*}
\begin{aligned}
&\left\lceil{\frac{8G^2}{\rho_k(1-\beta)^2(\Delta_k - \epsilon_1 - \epsilon_2)} - \frac{16}{(1-\beta)^2}}\right\rceil + 1 \leq \left\lceil{\frac{8G^2}{\rho_k(1-\beta)^2\left(\frac{\mu^2}{2\rho_k} - \epsilon_1 - \epsilon_2\right)} - \frac{16}{(1-\beta)^2}}\right\rceil + 1 \\
\leq \ &\frac{16G^2}{(1-\beta)^2\left(\mu^2 - \frac{2\rho_k\beps}{\beta+1} \right)} = \frac{16G^2}{(1-\beta)^2\mu^2\left(1 - \frac{2\beps}{(\beta+1)(\hf_k(\xk{0}) - \epsilon_2 - f^*)} \right)} \\
< \ &\frac{16G^2}{(1-\beta)^2\mu^2\left(1 - \frac{2\beta}{3(\beta+1)} \right)} = \mathcal{O}(1),
\end{aligned}
\end{equation*}
where the first equality follows the definition of $\rho_k$ and the last inequality uses \eqref{eq.sharp_ideal_func_gap_lb}. After combining all the components above, we may conclude the statement.
\end{proof}

\subsection{Supplemental Proofs for Section~\ref{sec.sample_const}}

We present the proof of
Theorem~\ref{thm:const_sample} here.
\begin{proof}[Proof of Theorem~\ref{thm:const_sample}]

Using Assumption~\ref{ass:sample_fg}, we first prove that $\hf_k(x)$ also has bounded variance and follows a sub-Gaussian distribution. 

Assumption~\ref{ass:sample_fg} states that the variance of $f(x;\xi)$ is no larger than $\sigma^2$. It follows from \eqref{eq:estimation} and \eqref{eq:f_xi} that for all $k\geq 0$ and $x\in X$,
\begin{equation}\label{eq:sample_variance_decrease}
\mathbb{E}[|\hf_k(x) - f(x)|^2] = \mathbb{E}\left[\left|\frac{1}{|S_k|}\sum_{i\in S_k}(f(x;\xi_i) - f(x))\right|^2\right] = \frac{1}{|S_k|^2}\cdot\sum_{i\in S_k}\mathbb{E}[|f(x;\xi_i) - f(x)|^2] \leq \frac{\sigma^2}{|S_k|},
\end{equation}
where the second equality is because $\xi_i$ comes from an independent distribution. Using the same logic above, by \eqref{eq:subGau}, we also have for all $k \geq 0$ and $x\in X$ that
\begin{align}\label{eq:subGau_decrease}
\mathbb{P}\left[|\hf_k(x) - f(x)| \geq \frac{\tau\sigma}{\sqrt{|S_k|}}\right] \leq 2{\rm exp}\left(-\frac{\tau^2}{2}\right)\quad \forall \tau > 0.
\end{align}
From the theorem statement, we know that given any parameter $\tilde{\epsilon}\in (0,1)$ and
\begin{equation}\label{eq:sample_size_requirement}
|S_k| = \max\left\{1, \frac{\sigma^2}{\tilde{\epsilon}^2}\right\},
\end{equation}
\eqref{eq:sample_variance_decrease} and \eqref{eq:subGau_decrease} directly implies that for all $k\geq 0$ and $x\in X$,
\begin{equation}\label{eq:fk_sample_subGauss}
\mathbb{E}[\epsilon_k(x)^2] \leq \tilde{\epsilon}^2\ \ \text{and}\ \  \mathbb{P}\left[\epsilon_k(x) \geq \tau\tilde{\epsilon}\right] \leq 2{\rm exp}\left(-\frac{\tau^2}{2}\right)\ \  \forall \tau > 0, \ \ \text{where } \epsilon_k(x) = |\hf_k(x) - f(x)|.
\end{equation}
\eqref{eq:fk_sample_subGauss} shows that when we choose a small value for parameter $\tilde\epsilon$, $\hf_k(x)$ has a small variance and is close to $f(x)$ with high probability. It's worth mentioning that \eqref{eq:fk_sample_subGauss} plays an important role in the rest of the proof, which includes three parts as follows.

\textbf{Part 1 (Length of the inner loop)} \quad Without loss of generality, let's suppose current outer iteration $k\geq 0$ satisfies $f(\xk{0}) - f^* > {\delta}^S$. Then following the same logic as the first part in the proof of Theorem~\ref{thm:const_main}, we know everything until \eqref{eq:decay_2} holds. 

Moreover, by \eqref{eq.inner_prox_gap} and the definition of $\xk{T_k}$, we have that
\begin{equation}\label{eq:error_sample}
\begin{aligned}
    \tDeltak{T_k-1} &= \hf_k(\xk{0})-\pran{\fk{T_k-1}(\xk{T_k})+\frac{\rho_{k}}{2}\|\xk{T_k}-\xk{0}\|^2} \\
    &= \hf_k(\xk{0})-\min_{x\in X}\left\{\fk{T_k-1}(x) + \frac{\rho_{k}}{2}\|x-\xk{0}\|^2\right\} \\
    &\geq (f(\xk{0}) - \epsilon_k(\xk{0})) - \left(\fk{T_k-1}(\bar{x}_{k+1}) + \frac{\rho_{k}}{2}\|\bar{x}_{k+1} - \xk{0}\|^2\right) \\
    &\geq  (f(\xk{0}) - \epsilon_k(\xk{0})) - \left(f(\bar{x}_{k+1}) + \epsilon_k(\bar{x}_{k+1}) + \frac{\rho_{k}}{2}\|\bar{x}_{k+1} - \xk{0}\|^2\right) \\
    &= \Delta_{k}-\epsilon_k(\xk{0})-\epsilon_k(\bar{x}_{k+1}),
\end{aligned}
\end{equation}
where the first inequality is from \eqref{eq:fk_sample_subGauss}, the second inequality follows Assumption~\ref{ass:bundle_func} and \eqref{eq:fk_sample_subGauss}, and the last equality uses \eqref{eq.Delta_k}. Using the same logic as the proof of Theorem~\ref{thm:const_main}, the number of inner iterations at the $k$-th outer iteration could be upper bounded by 
\begin{equation}\label{eq.Tk_bound_sample}
\begin{aligned}
T_k - 1 \leq \left\lceil{\frac{\frac{1}{\Delta_k -\epsilon_k(\xk{0})-\epsilon_k(\bar{x}_{k+1})} - \frac{2\rho_k}{G^2}}{\frac{\rho_k(1-\beta)^2}{8G^2}}}\right\rceil \leq \left\lceil{\frac{8G^2}{\rho_k(1-\beta)^2(\Delta_k - \epsilon_k(\xk{0})-\epsilon_k(\bar{x}_{k+1}))} - \frac{16}{(1-\beta)^2}}\right\rceil,
\end{aligned}
\end{equation}
which provides an upper bound for $T_k$.

\textbf{Part 2 (Objective decrease on serious iterates)} \quad Similar to the logic of the second part in the proof of Theorem~\ref{thm:const_main}, by \eqref{eq:fk_sample_subGauss} and the definition of serious step, we have
\begin{equation}\label{eq.obj_decay_sample}
\begin{aligned}
    f(x_{k+1,0})&\le \hf_{k}(x_{k+1,0})+\epsilon_k(x_{k+1,0}) = \hf_{k}(x_{k,T_k})+\epsilon_k(x_{k+1,0}) \\
    &\le \hf_k(\xk{0}) - \beta(\hf_k(\xk{0})-\fk{T_k-1}(\xk{T_k}))+\epsilon_k(x_{k+1,0})\\
    &\le \hf_k(\xk{0}) - \beta\pran{\hf_k(\xk{0})-\pran{\fk{T_k-1}(\xk{T_k})+\frac{\rho_k}{2}\|\xk{T_k}-\xk{0}\|^2}}+\epsilon_k(x_{k+1,0})\\
    &\le \hf_k(\xk{0}) - \beta\pran{\hf_k(\xk{0})-\pran{\fk{T_k - 1}(\bx_{k+1})+\frac{\rho_k}{2}\|\bx_{k+1}-\xk{0}\|^2}}+\epsilon_k(x_{k+1,0})\\
    &\le \hf_k(\xk{0}) - \beta\pran{(f(\xk{0}) - \epsilon_k(\xk{0}))-\pran{f(\bx_{k+1})+\frac{\rho_k}{2}\|\bx_{k+1}-\xk{0}\|^2+\epsilon_k(\bx_{k+1})}}+\epsilon_k(x_{k+1,0})\\
    & = \hf_k(\xk{0}) - \beta\Delta_k+\beta(\epsilon_k(\xk{0})+\epsilon_k(\bx_{k+1}))+\epsilon_k(x_{k+1,0}) \\
    & \le f(\xk{0}) - \beta\Delta_k + (\beta+1)\epsilon_k(\xk{0})+\beta\epsilon_k(\bx_{k+1})+\epsilon_k(x_{k+1,0}) \\
    & = f(\xk{0})-\left(\beta \Delta_k - \beps_k \right),
\end{aligned}
\end{equation}
where the fourth inequality follows the update rule of $x_{k,T_k}$, the fifth inequality is from Assumption~\ref{ass:bundle_func} and \eqref{eq:fk_sample_subGauss}, the second equality follows \eqref{eq.Delta_k}, and in the last equality we define
\begin{equation}\label{eq:beps_k}
\beps_k = (\beta+1)\epsilon_k(\xk{0})+\beta\epsilon_k(\bx_{k+1})+\epsilon_k(x_{k+1,0}).
\end{equation}

\textbf{Part 3 (Total outer and inner iterations)} \quad 
We note that for any outer iteration $k\geq 0$, it follows \eqref{eq:fk_sample_subGauss} that
\begin{align}\label{eq:bounded_err_prob}
\mathbb{P}\left[\max\left\{\epsilon_k(\xk{0}), \epsilon_k(\bar{x}_{k+1}),\epsilon_k(x_{k+1,0})\right\}< \tau\tilde{\epsilon}\right] \geq 1-6{\rm exp}\left(-\frac{\tau^2}{2}\right)\quad \forall \tau > 0.
\end{align}
In the rest of this proof, we only consider the case that $\max\left\{\epsilon_k(\xk{0}), \epsilon_k(\bar{x}_{k+1}),\epsilon_k(x_{k+1,0})\right\}< \tau\tilde{\epsilon}$ holds for all $k < K^S$, from \eqref{eq:bounded_err_prob}, which happens with probability at least $1-6K^S{\rm exp}\left(-\frac{\tau^2}{2}\right)$.

Next, we prove the theorem by contradiction. Let's suppose for all outer iterations $k \leq K^S$ it holds that $f(\xk{0})-f^* > {\delta}^S$. By the step size policy and \eqref{eq:fk_sample_subGauss}, we know 
\begin{align*}
\rho_kD^2 = \hf_k(\xk{0})-\tau\tilde{\epsilon}-f^* < \hf_k(\xk{0})- \epsilon_k(\xk{0}) - f^* \leq f(\xk{0}) - f^*.
\end{align*}
Then we have from \eqref{eq:proximal_bound} that for all $k\geq 0$,
\begin{equation}\label{eq:sample-Deltak}
\begin{aligned}
    \Delta_k \ge  \frac{f(\xk{0})-f^*}{2}. 
\end{aligned}
\end{equation}
Combining \eqref{eq.obj_decay_sample} and \eqref{eq:sample-Deltak}, we have 
\begin{equation}\label{eq:sample_func_decay_ideal}
f(x_{k+1,0}) - f^* \leq \left(1-\frac{\beta}{2}\right)\cdot(f(\xk{0}) - f^*) + \beps_k.
\end{equation}
Meanwhile, from \eqref{eq:beps_k} and the pre-assumption that $\max\left\{\epsilon_k(\xk{0}), \epsilon_k(\bar{x}_{k+1}),\epsilon_k(x_{k+1,0})\right\}< \tau\tilde{\epsilon}$ holds for all $k < K^S$, we know for every outer iteration $k < K^S$,
\begin{equation}\label{eq:sample_prob_intermediate_1}
\beps_k \leq 2(\beta+1)\tau\tilde{\epsilon}.
\end{equation} 
Thus, it follows \eqref{eq:sample_func_decay_ideal} and \eqref{eq:sample_prob_intermediate_1} that for any $k < K^S$, 
\begin{equation}\label{eq:sample_ideal_suff_decay}
\begin{aligned}
f(x_{k+1,0}) - f^* &\leq \left(1-\frac{\beta}{2}\right)\cdot(f(\xk{0}) - f^*) +\beps_k \\
&\leq \left(1-\frac{\beta}{2}\right)\cdot(f(\xk{0}) - f^*) + 2(\beta+1)\tau\tilde{\epsilon} \\
&< \left(1-\frac{\beta}{4}\right)\cdot(f(\xk{0}) - f^*),
\end{aligned}
\end{equation}
where the last inequality follows \eqref{eq:const_neighborhood_sample} that $f(\xk{0})-f^* > {\delta}^S = \frac{8(\beta+1)\tilde{\epsilon}\tau}{\beta}$.
Therefore, by considering all outer iterations $k < K^S$, from \eqref{eq:const_neighborhood_sample} we know with a probability at least $1-6K^S{\rm exp}\left(-\frac{\tau^2}{2}\right)$ that
\begin{align*}
f(x_{K^S,0})-f^* < \left(1-\frac{\beta}{4}\right)^{K^S}(f(x_{0,0}) - f^*) \leq {\delta}^S,
\end{align*}
which results in a contradiction. Therefore, with a probability at least $1-6K^S{\rm exp}\left(-\frac{\tau^2}{2}\right)$, there must exist at least one $k\leq K^S$ such that $f(\xk{0}) - f^* \leq {\delta}^S$. 

Let's denote $K^S_{\max}\geq 0$ as the first outer iteration such that $f(x_{K^S_{\max},0}) - f^* > {\delta}^S$ and $f(x_{K^S_{\max}+1,0}) - f^* \leq {\delta}^S$. Under the pre-assumption of $\max\{\epsilon_k(\xk{0}), \epsilon_k(\bar{x}_{k+1}),\epsilon_k(x_{k+1,0})\}< \tau\tilde{\epsilon}$ for all $k<K^S$, which happens with a probability at least $1-6K^S{\rm exp}\left(-\frac{\tau^2}{2}\right)$, we have $\tilde{K}_{\max}^I < K^S$, and it also follows \eqref{eq:sample-Deltak} that 
\begin{equation}\label{eq:sample_Deltak_lb_1}
\Delta_{K^S_{\max}} > \frac{f(x_{K^S_{\max},0})-f^*}{2} > \frac{{\delta}^S}{2} > 2\tau\tilde{\epsilon} + 2\tau\tilde{\epsilon} > 2\tau\tilde{\epsilon} + \epsilon_{K^S_{\max}}(x_{K^S_{\max},0})+\epsilon_{K^S_{\max}}(\bar{x}_{K^S_{\max}+1}),
\end{equation}
where the third inequality follows \eqref{eq:const_neighborhood_sample}. Moreover, under the condition of $\epsilon_{K^S_{\max}}(x_{K^S_{\max},0}) < \tau\tilde{\epsilon}$, it follows from the step size policy and \eqref{eq:fk_sample_subGauss} that 
\begin{equation}\label{eq:sample_rho_Kmax_lb}
\rho_{K^S_{\max}} = \frac{\hf_{K^S_{\max}}(x_{K^S_{\max},0})-\tau\tilde{\epsilon}-f^*}{D^2} \geq \frac{f(x_{K^S_{\max},0}) - f^* - \epsilon_{K^S_{\max}}(x_{K^S_{\max},0}) - \tau\tilde{\epsilon}}{D^2} > \frac{{\delta}^S - 2\tau\tilde{\epsilon}}{D^2} > \frac{6\tau\tilde{\epsilon}}{\beta D^2},
\end{equation}
where the last inequality is from \eqref{eq:const_neighborhood_sample}. Furthermore, from \eqref{eq:sample_ideal_suff_decay} and the same logic in \eqref{eq:sample_Deltak_lb_1} and \eqref{eq:sample_rho_Kmax_lb}, with probability at least $1-6K^S{\rm exp}\left(-\frac{\tau^2}{2}\right)$ that for any $k\in\{0,\ldots,\tilde{K}_{\max}^I\}$,
\begin{equation}\label{eq:sample_lb_Deltak_rhok}
\begin{aligned}
\Delta_{k}& \ge \frac{1}{2}(f(x_{k,0})-f^*) \geq \frac{f(x_{K^S_{\max},0})-f^*}{2(1-\frac{\beta}{4})^{(K^S_{\max}-k)}} > \frac{{\delta}^S}{2(1-\frac{\beta}{4})^{(K^S_{\max}-k)}} \\
&= \frac{4(\beta+1)\tilde\epsilon\tau}{\beta(1-\frac{\beta}{4})^{(K^S_{\max}-k)}} > \frac{2\tilde\epsilon\tau}{(1-\frac{\beta}{4})^{(K^S_{\max}-k)}} + \epsilon_k(\xk{0}) +\epsilon_k(\bar{x}_{k+1}), \\
\text{and}\quad \rho_{k} &= \frac{\hf_{k}(x_{k,0})-\tau\tilde\epsilon-f^*}{D^2} \geq \frac{f(x_{k,0}) - f^* - \epsilon_k(\xk{0}) -\tau\tilde\epsilon}{D^2} \\
&\geq \frac{\frac{f(x_{K^S_{\max},0})-f^*}{(1-\frac{\beta}{4})^{(K^S_{\max}-k)}} - \epsilon_k(\xk{0}) -\tau\tilde\epsilon}{D^2} > \frac{\frac{{\delta}^S}{(1-\frac{\beta}{4})^{(K^S_{\max}-k)}} - \epsilon_k(\xk{0}) -\tau\tilde\epsilon}{D^2} \\
&> \frac{6\tau\tilde\epsilon}{\beta(1-\frac{\beta}{4})^{(K^S_{\max}-k)} D^2}.
\end{aligned}
\end{equation}
Therefore, by \eqref{eq.Tk_bound_sample} and \eqref{eq:sample_lb_Deltak_rhok}, we know with probability at least $1-6K^S{\rm exp}\left(-\frac{\tau^2}{2}\right)$, the total number of null steps is at most
\begin{align*}
&\sum_{k=0}^{K^S_{\max}} \left(\left\lceil{\frac{8G^2}{\rho_k(1-\beta)^2(\Delta_k - \epsilon_k(\xk{0})-\epsilon_k(\bar{x}_{k+1}))} - \frac{16}{(1-\beta)^2}}\right\rceil + 1\right) \\
< \ &\sum_{k=0}^{K^S_{\max}} \left(\left\lceil{\frac{8G^2}{\frac{6\tau\tilde\epsilon}{\beta(1-\frac{\beta}{4})^{(K^S_{\max}-k)} D^2}\cdot(1-\beta)^2\cdot\frac{2\tau\tilde\epsilon}{(1-\frac{\beta}{4})^{(K^S_{\max}-k)}}} - \frac{16}{(1-\beta)^2}}\right\rceil + 1\right) \\
= \ &\sum_{k=0}^{K^S_{\max}} \left(\left\lceil{\frac{2\beta(1-\frac{\beta}{4})^{2(K^S_{\max}-k)}G^2D^2}{3(1-\beta)^2\tau^2\tilde\epsilon^2} - \frac{16}{(1-\beta)^2}}\right\rceil + 1\right) \\
\leq \ &\sum_{k=0}^{K^S_{\max}} \left(\frac{2\beta(1-\frac{\beta}{4})^{2(K^S_{\max}-k)}G^2D^2}{3(1-\beta)^2\tau^2\tilde\epsilon^2}\right) \\
= \ &\frac{2\beta G^2D^2}{3(1-\beta)^2\tau^2\tilde\epsilon^2}\cdot\frac{1-(1-\frac{\beta}{4})^{2(K^I_{\max}+1)}}{1-(1-\frac{\beta}{4})^2} < \frac{8G^2D^2}{3(1-\beta)^2(2-\frac{\beta}{4})\tau^2\tilde\epsilon^2},
\end{align*}
which concludes the statement.
\end{proof}

Based on the proof of Theorem~\ref{thm:const_sample}, we present the proof of Corollary~\ref{cor:sample_complexity} by choosing suitable parameters.

\begin{proof}[Proof of Corollary~\ref{cor:sample_complexity}] 
When $f(x_{0,0}) - f^* \leq \delta^S$, then the initial iterate is already $\delta^S$-optimal. Therefore, in the rest of this proof, we only consider the case that $f(x_{0,0}) - f^* > \delta^S$.
For any $\delta^S \in (0,1)$, following the corollary statement, we set
\begin{equation}\label{eq:sample_tau_teps_selection}
\begin{aligned}
\tau &= \max\left\{1,\sqrt{2\log\left(\frac{\max\left\{1,6\cdot\left\lceil\log_{\left(1-\frac{\beta}{4}\right)}\left(\frac{\delta^S}{f(x_{0,0}) - f^*}\right)\right\rceil \right\}}{\zeta}\right)}\right\}  = \mathcal{O}\left(\sqrt{\log\left(\log\left(\frac{1}{\delta^S}\right)\right)}\right) \\
\text{and}\ \ \tilde\epsilon &= \frac{\beta\delta^S}{8(\beta+1)\tau} = \mathcal{O}(\delta^S).
\end{aligned}
\end{equation}
Then we know from \eqref{eq:sample_tau_teps_selection} and \eqref{eq:const_neighborhood_sample} that 
\begin{align*}
1-6K^S{\rm exp}\left(-\frac{\tau^2}{2}\right) &\geq 1-6K^S{\rm exp}\left(-\log\left(\frac{\max\left\{1,6\cdot\left\lceil\log_{\left(1-\frac{\beta}{4}\right)}\left(\frac{\delta^S}{f(x_{0,0}) - f^*}\right)\right\rceil \right\}}{\zeta}\right)\right) \\
&= 1-\frac{6K^S\zeta}{\max\left\{1,6\cdot\left\lceil\log_{\left(1-\frac{\beta}{4}\right)}\left(\frac{\delta^S}{f(x_{0,0}) - f^*}\right)\right\rceil \right\}} \\
&= 1-\frac{\zeta\cdot\max\left\{0, 6\cdot\left\lceil\log_{\left(1-\frac{\beta}{4}\right)}\left(\frac{\delta^S}{f(x_{0,0}) - f^*}\right)\right\rceil\right\}}{\max\left\{1,6\cdot\left\lceil\log_{\left(1-\frac{\beta}{4}\right)}\left(\frac{\delta^S}{f(x_{0,0}) - f^*}\right)\right\rceil \right\}} \geq 1-\zeta.
\end{align*}
Moreover, it follows \eqref{eq:total-null-const} and \eqref{eq:sample_tau_teps_selection} that with probability at least $1-\zeta$, Algorithm~\ref{alg.main} with the optimal step size policy needs at most 
\begin{equation}\label{eq:sample_proof_inner_ub}
\frac{8G^2D^2}{3(1-\beta)^2(2-\frac{\beta}{4})\tau^2\tilde\epsilon^2} = \frac{256G^2D^2(\beta+1)^2}{3(1-\beta)^2(2-\frac{\beta}{4})\beta^2\left(\delta^S\right)^2} = \mathcal{O}\left(\left(\frac{1}{\delta^S}\right)^2\right)
\end{equation}
number of bundle steps (including null steps) to find an $\delta^S$-optimal iterate. By \eqref{eq:sample_proof_inner_ub}, \eqref{eq:sample_size_requirement}, and \eqref{eq:sample_tau_teps_selection}, we further know that with probability at least $1-\zeta$, at most 
\begin{align*}
&\max\left\{1, \frac{\sigma^2}{\tilde{\epsilon}^2}\right\} \cdot \frac{256G^2D^2(\beta+1)^2}{3(1-\beta)^2(2-\frac{\beta}{4})\beta^2\left(\delta^S\right)^2}\\
= \ &\max\left\{\frac{256G^2D^2(\beta+1)^2}{3(1-\beta)^2(2-\frac{\beta}{4})\beta^2\left(\delta^S\right)^2}, \frac{256\sigma^2G^2D^2(\beta+1)^2}{3(1-\beta)^2(2-\frac{\beta}{4})\beta^2\tilde\epsilon^2\left(\delta^S\right)^2}\right\} \\
= \ &\max\left\{\frac{256G^2D^2(\beta+1)^2}{3(1-\beta)^2(2-\frac{\beta}{4})\beta^2\left(\delta^S\right)^2}, \frac{16384\sigma^2G^2D^2(\beta+1)^4\tau^2}{3(1-\beta)^2(2-\frac{\beta}{4})\beta^4\left(\delta^S\right)^4}\right\} \\
= \ &\mathcal{O}\left(\frac{\log\left(\log\left(\frac{1}{\delta^S}\right)\right)}{\left(\delta^S\right)^4}\right)
\end{align*}
number of $f(x;\xi)$ samples are needed to find an $\delta^S$-optimal iterate. 
\end{proof}

\end{document}